 \documentclass[11pt,reqno]{amsart}
\usepackage{epsfig,amssymb,amsmath,version}
\usepackage{amssymb,version,graphicx,fancybox,mathrsfs,bm}
\usepackage{url,hyperref}
\usepackage{subfigure}
\usepackage{color}
\usepackage{stmaryrd}
\usepackage{multirow}
\usepackage{booktabs,siunitx}

\usepackage{booktabs,enumerate}
\usepackage[ruled,vlined]{algorithm2e}

\catcode`\@=11 \theoremstyle{plain}
\@addtoreset{equation}{section}   
\renewcommand{\theequation}{\arabic{section}.\arabic{equation}}
\@addtoreset{figure}{section}
\renewcommand\thefigure{\@arabic\c@figure}
\newtheorem{thm}{\bf Theorem}

\newenvironment{theorem}{\begin{thm}} {\end{thm}}
\newtheorem{cor}{\bf Corollary}

\newtheorem{lmm}{\bf Lemma}

\newenvironment{lemma}{\begin{lmm}}{\end{lmm}}
\theoremstyle{remark}
\newtheorem{rem}{\bf Remark}

\def \epsilon {{\varepsilon}}

\definecolor{bgblue}{rgb}{0.04,0.39,0.54}
\definecolor{lired}{rgb}{0.3, 0.0, 0.0}
\definecolor{ligreen}{rgb}{0.0, 0.3, 0.0}
\definecolor{liblue}{rgb}{0.9, 1.0, 1.0}
\definecolor{gray}{rgb}{0.6, 0.6, 0.6}
\definecolor{sky}{rgb}{0.3, 1.0, 1.0}
\definecolor{bunhong}{rgb}{1.0, 0.3, 1.0}
\definecolor{yellow}{rgb}{0.97, 1, 0.0}
\definecolor{liyellow}{rgb}{0.9, 0.8, 0.0}
\definecolor{cengse}{rgb}{0.00,0.40,0.29}

\renewcommand \wedge \times

\graphicspath{{../figsreply/}}

\begin{document}
\bibliographystyle{plain}
{\title[Implicit-explicit BDF$k$ SAV schemes] {Implicit-explicit BDF$k$ SAV schemes for general dissipative systems and their error analysis}
\author[
	F. Huang and J. Shen
	]{
	Fukeng Huang and Jie Shen
		}
	\thanks{Department of Mathematics, Purdue University. This research  is partially supported by NSF  DMS-2012585
and  AFOSR FA9550-20-1-0309. Emails: huang972@purdue.edu (F. Huang), shen7@purdue.edu (J. Shen).}


\begin{abstract}
We construct  efficient implicit-explicit BDF$k$   scalar auxiliary variable (SAV) schemes  for general dissipative systems. We show that these schemes are unconditionally  stable, and lead to a uniform bound of the numerical solution in the norm based on the  principal linear operator in the energy. Based on this uniform bound, we carry out a rigorous error analysis for the $k$th-order $(k=1,2,3,4,5)$ SAV schemes  in a unified form for a class of typical Allen-Cahn type and Cahn-Hilliard type equations. We also present numerical results confirming our theoretical convergence rates.

\end{abstract}
 \maketitle
\section{Introduction} The original scalar auxiliary variable (SAV) approach   proposed in \cite{SXY18, SXY19} is a powerful approach to construct efficient time discretization schemes for  gradient flows. Due to its simplicity, efficiency and generality, it attracted much attention and has been applied to various problems (see, for instance, \cite{Shen19,She.Y20} and the references therein).

While analysis of standard semi-implicit schemes for gradient flows usually requires to assume  global Lipschitz condition on the nonlinear term (see, for instance, \cite{kessler2004posteriori,shen2010numerical,condette2011spectral}),
 the convergence of SAV schemes can be established without such assumption thanks to the unconditional energy stability. For examples,
 rigorous error analysis of the semi-discretized first order original SAV schemes for  ${L}^2$ and  $H^{-1}$ gradient flows with minimum assumptions have been presented in \cite{SX18}, first- and second-order error estimates have been derived for a related semi-discretized g{PAV} scheme for the Cahn-Hilliard equation in \cite{qian2020gpav}, and  error analysis of fully discretized SAV schemes with finite differences and finite-elements have also been established in \cite{LSR19} and \cite{CMS20}.   On the other hand, error estimates for a  Fourier-spectral  SAV scheme for the phase-field crystal equation \cite{Li.S20}
 and a MAC-SAV scheme for the Navier-Stokes equation  \cite{Li.S20SINUM} are  established.
 Note that for the original SAV approach,  unconditional energy stability can  only be established for first- and second-order BDF schemes, although it has been shown in \cite{ALL19} (see also \cite{DLi.S20}) that the SAV approach coupled with  extrapolated and linearized  Runge-Kutta methods  can achieve arbitrarily high order unconditionally energy stable with a modified energy for the Allen-Cahn and Cahn-Hilliard equations, but  require solving coupled linear systems.

Most recently, a new SAV approach for gradient flow is proposed in \cite{HSY20} which offers some essential improvements over the original SAV approach such as (i) its computational cost is about half of the original SAV approach, and (ii) its higher-order BDF versions are also unconditionally stable with a modified energy. While ample numerical results in \cite{HSY20} have shown that the new higher-order SAV schemes are indeed stable and can achieve higher-order accuracy, the modified energy is represented only by a SAV which does not involve any function norm, so it is difficult to carry out convergence and error analysis. See however \cite{qian2020gpav} for an attempt on the error analysis for related  first- and second-orer gPAV-based schemes for the Cahn-Hilliard equation.

In this paper, we apply the general ideas in \cite{HSY20} to construct a class of explicit-implicit BDF$k$ schemes for general dissipative systems. In particular, we choose a special control factor $\eta_k^{n+1}$ (cf.  \eqref{eq: Nsav2}) for the $k$th-order scheme which allows us to obtain a unconditional and uniform bound on the  norm based on principal linear term in  the energy functional of the dissipative system. This bound is essential for the  error analysis in this paper.
The main purpose of this paper is to carry out a rigorous and unified error analysis for the  $k$th-order ($1\le k\le 5$) SAV schemes which enjoy  several  advantages, including:
\begin{itemize}
\item  only requires solving, in most common situations,  one linear system with constant coefficients at each time step, so its computational cost is essentially the same as the usual implicit-explicit (IMEX) schemes;
\item  applicable to  general dissipative systems;
\item  higher-order BDF$k$ SAV schemes are unconditionally stable and amenable to adaptive time stepping without restriction on time step size;
\item rigorous error estimates can be established for  BDF$k$ ($1\le k\le 5$) SAV schemes.
\end{itemize}

While these SAV schemes are applicable to a large class of dissipative systems,  their  error analysis is highly non-trivial, particularly at higher than second order. Since  a unified analysis for general dissipative systems will  involve  complicated assumptions and techniques that may obscure the clarity of presentation, we shall consider the error analysis for two classes of typical dissipative systems: Allen-Cahn type and Cahn-Hilliard type equations. The key ingredients are  the  uniform $H^1$ bound derived from the general stability result (see \eqref{eq: bound} in Theorem \ref{stableThm}) and a  stability result in \cite{nevanlinna1981} (see Lemma \ref{lemmaODEH} below) for the BDF$k$ ($1\le k\le 5$) schemes. With a delicate induction argument, we are able to establish optimal error estimates in $L^\infty(0,T;H^2)$ norm for our implicit-explicit BDF$k$ ($1\le k\le 5$) SAV  schemes for both Allen-Cahn type and Cahn-Hilliard type equations. 

The rest of the paper is organized as follows. In the next section, we describe out  new SAV schemes for general dissipative systems in a unified form,  prove its unconditionally stability, and provide some numerical results to demonstrate the convergence rate. In section 3, we present the detailed proof for the $k$th-order schemes $(k=1,2,3,4,5)$ in a unified form for Allen-Cahn type equations. In section 4, we present  error analysis for Cahn-Hilliard type equations.  Some concluding remarks are given in the last section.

We use the following notations throughout the paper. Let  $\Omega \in \mathcal{R}^n\; (n=1,2,3)$ be a bounded domain with sufficiently smooth boundary. We denote by $(\cdot, \cdot)$ and $\|\cdot\|$ the inner product and the norm in $L^2(\Omega)$, and by $H^s(\Omega)$ the usual Sobolev spaces with norm $\|\cdot\|_{H^s}$. Let $V$ be a Banach space, we shall also use the standard notations
  $L^p(0,T;V)$ and $C([0,T];V)$. To simplify the notation, we often omit the spatial dependence in the notation for the exact solution $u$, namely we denote $u(x,t)$ by $u(t)$.  We shall use $C$ to denote a constant which can change from one step to another, but is independent of $\delta t$.

\section{New SAV schemes  for dissipative systems}
In this section, we describe  the new SAV schemes for dissipative systems, show that they are unconditionally energy stable with a modified energy and derive a uniform bound for the norm based on the principal linear term in the energy functional.

Consider the following class of   dissipative systems
\begin{equation}\label{dissE}
 \frac{\partial u}{\partial t} +\mathcal{A} u+g(u)=0,
\end{equation}
where $u$ is a scalar or vector function, $\mathcal{A}$ is a positive differential operator and $g(u)$ is a nonlinear operator possibly with lower-order derivatives. We assume that the above equation satisfies
a dissipative energy law
\begin{equation}\label{dissL}
\frac{d \tilde E(u)}{dt}=-\mathcal{K}(u),
\end{equation}
where $ \tilde E(u)> -C_0$ for all $u$ is an energy functional,  $\mathcal{K}(u)> 0$ for all $u\ne 0$.

The above class of dissipative systems include in particular gradient flows but also other dissipative systems which do not have the gradient structure, such as viscous Burgers equation, reaction-diffusion equations etc.

\subsection{The new SAV schemes}
The key for the SAV approach is to introduce a scalar auxiliary variable (SAV) to rewrite  \eqref{dissE} as an expanded system, and to discretize the expanded system instead of the original \eqref{dissE}. In this paper, we introduce the following new SAV approach inspired by the SAV schemes introduced in \cite{HSY20}

Setting $r(t)=E(u)(t):=\tilde E(u)(t)+C_0>0$,  we  rewrite the equation \eqref{dissE}  with the energy law \eqref{dissL} as the following expanded system
\begin{align}
& \frac{\partial u}{\partial t}+\mathcal{A} u +g(u)=0, \label{eq: new1}\\
&\frac{d E(u)}{dt} =-\frac{r(t)}{E(u)(t)}\mathcal{K}(u).\label{eq: new2}
\end{align}

We construct the $k$th  order new SAV schemes based on the implicit-explicit BDF-$k$ formulae in the following unified form:

Given $u^n,\,r^n$, we compute $\bar u^{n+1},\, r^{n+1}, \, \xi^{n+1}$ and $u^{n+1}$ consecutively by
\begin{subequations}\label{eq: Nsav}
\begin{align}
&\frac{\alpha_k \bar{u}^{n+1}-A_k(u^n)}{ \delta t}+\mathcal{A} \bar u^{n+1}+g[B_k(\bar{u}^n)])=0,\label{eq: Nsav1}\\
& \frac{1}{\delta t}\big(r^{n+1}-r^{n}\big)=-\frac{r^{n+1}}{E(\bar{u}^{n+1})}\mathcal{K}(\bar{u}^{n+1}),\label{eq: Nsav3}\\
& \xi^{n+1}=\frac{r^{n+1}}{E(\bar{u}^{n+1})},\label{eq: Nsav4}\\
&u^{n+1}=\eta_{k}^{n+1}\bar{u}^{n+1} \;\text{ with } \eta^{n+1}_{k}=1-(1-\xi^{n+1})^{k+1}.\label{eq: Nsav2}
\end{align}
\end{subequations}
where  $\alpha_k,$ the operators $A_k$ and $B_k$ $(k=1,2,3,4,5)$ are given by:
\begin{itemize}
\item[first-order:]
\begin{equation}\label{eq:bdf1}
\alpha_1=1, \quad A_1(u^n)=u^n,\quad B_1(\bar {u}^n)=\bar {u}^n;
\end{equation}
\item[second-order:]
\begin{equation}\label{eq:bdf2}
\alpha_2=\frac{3}{2}, \quad A_2(u^n)=2u^n-\frac{1}{2}u^{n-1},\quad B_2(\bar {u}^n)=2\bar {u}^n-\bar {u}^{n-1};
\end{equation}
\item[third-order:]
\begin{equation}\label{eq:bdf3}
\alpha_3=\frac{11}{6}, \quad A_3(u^n)=3u^n-\frac{3}{2}u^{n-1}+\frac{1}{3}u^{n-2},\quad B_3(\bar {u}^n)=3\bar {u}^n-3\bar {u}^{n-1}+\bar u^{n-2};
\end{equation}
\item[fourth-order:]
\begin{equation}\label{eq:bdf4}
\alpha_4=\frac{25}{12}, \; A_4(u^n)=4u^n-3u^{n-1}+\frac{4}{3}u^{n-2}-\frac{1}{4}u^{n-3},\; B_4(\bar {u}^n)=4\bar {u}^n-6\bar {u}^{n-1}+4\bar {u}^{n-2}-\bar {u}^{n-3}.\end{equation}
\item[fifth-order:]
\begin{equation}\label{eq:bdf5}
\begin{split}
\alpha_5=\frac{137}{60}, \quad & A_5(u^n)=5u^n-5u^{n-1}+\frac{10}{3}u^{n-2}-\frac{5}{4}u^{n-3}+\frac{1}{5}u^{n-4},\\
& B_5(\bar {u}^n)=5\bar {u}^n-10\bar {u}^{n-1}+10\bar {u}^{n-2}-5\bar {u}^{n-3}+\bar {u}^{n-4}.
\end{split}\end{equation}
\end{itemize}
Several remarks are in order:
\begin{itemize}
\item Initialization: the second-order scheme can be initialized with a first-order scheme for the first step, the $k$th-order scheme can be initialized with a $k-1$th-order Runge-Kutta method for the first $k-1$ steps.
\item  We observe from \eqref{eq: Nsav3} that $r^{n+1}$ is a first order approximation to $E(u(\cdot,t_{n+1}))$ which implies that $\xi^{n+1}$ is a  first order approximation to 1.
\item
we observe from \eqref{eq: Nsav1} and \eqref{eq: Nsav2} that
\begin{equation}\label{Nsav4}
 \frac{\alpha_k u^{n+1}-\eta_k^{n+1} A_k(u^n)}{ \delta t}+\mathcal{A} u^{n+1}+\eta_k^{n+1}g[B_k(\bar{u}^n)]=0,
\end{equation}
which, along with   \eqref{eq: Nsav2}, implies that
\begin{equation*}
 \frac{\alpha_k u^{n+1}-A_k(u^n)}{ \delta t}+\mathcal{A} u^{n+1}+g[B_k({u}^n)]=O(\delta t^k).
\end{equation*}
Hence, both $u^{n+1}$ and $\bar u^{n+1}$ are formally  $k$th order approximations for $u(\cdot, t^{n+1})$.

\item The main difference of the above scheme from the scheme in \cite{HSY20} is the choice of $\eta_k^{n+1}$, which can be considered as a special case in \cite{HSY20}. However, as we show below, this choice allows us to obtain a uniform bound on $(\mathcal{L}u^n,u^n)$, which in turn plays a crucial role in the error analysis. Another slight difference is here we use $g[B_k(\bar{u}^n)]$ in \eqref{eq: Nsav1}, which makes the error analysis slightly easier, while $g[B_k({u}^n)]$ is used in  \cite{HSY20}. Thanks to  \eqref{eq: Nsav2}, this does not affect the $k$th order accuracy nor unconditional energy stability.

\item Since the energy stability is achieved through only \eqref{eq: Nsav3}, we can replace \eqref{eq: Nsav1} by other types of explicit-implicit multistep schemes.

\end{itemize}

The above scheme can be efficiently implemented as follows:
\begin{enumerate}[i.]
\item Obtain $\bar u^{n+1}$ from \eqref{eq: Nsav1} by solving an equation of the form
$$(\frac{\alpha_k}{\delta t} I +\mathcal{A})\bar u^{n+1}=f^{n+1},$$
where $f^{n+1}$ includes all known terms from previous time steps, and in most cases,  this is a linear equation with constant coefficients;
\item  With $\bar u^{n+1}$  known, determine $r^{n+1}$ explicitly  from \eqref{eq: Nsav3};
\item Compute $\xi^{n+1}$, $\eta_k^{n+1}$ and  $u^{n+1}$ from \eqref{eq: Nsav2}, goto the next step.
\end{enumerate}
The main computational cost of this scheme is to solve \eqref{eq: Nsav1} once, while the main computational  cost in the original SAV approach is to solve an equation similar to \eqref{eq: Nsav1} twice. So the cost of this scheme is about half of the original SAV approach while enjoying the same unconditional energy stability as we show below.

\subsection{A stability result}
We have the following results concerning the stability of the above schemes.
\begin{theorem}\label{stableThm}
Given $r^n \ge 0$, we have $r^{n+1} \ge 0$, $\xi^{n+1} \ge 0$, and the scheme \eqref{eq: Nsav} for any $k$ is unconditionally energy stable in the sense that
\begin{equation}\label{eq: energystable}
r^{n+1}-r^{n}=-\delta t \xi^{n+1}\mathcal{K}(\bar{u}^{n+1}) \le 0.
\end{equation}
Furthermore,  if $E(u)=\frac 12 (\mathcal{L}u,u)+E_1(u)$ with $\mathcal{L}$ positive and  $E_1(u)$ bounded from below, there exists $M_k>0$ such that
\begin{equation}\label{eq: bound}
(\mathcal{L}u^n, u^n) \le M^2_k,\,\forall n.
\end{equation}
\end{theorem}
\begin{proof}
Given $r^n \ge 0$ and since $E[\bar{u}^{n+1}] >0$, it follows from \eqref{eq: Nsav3} that
\begin{equation*}
r^{n+1}=\frac{r^{n}}{1+\delta t\frac{(\mathcal{K}(\bar{u}^{n+1}) }{E[\bar{u}^{n+1}]}} \ge 0.
\end{equation*}
Then we derive from \eqref{eq: Nsav4} that $\xi^{n+1} \ge 0$ and obtain \eqref{eq: energystable}.

Denote $M:=r^0=E[u(\cdot,0)]$, then \eqref{eq: energystable} implies $r^n \le M,\, \forall n$.

Without loss of generality, we can assume $E_1(u)>1$ for all $u$. It then follows from \eqref{eq: Nsav4}  that
\begin{equation}\label{eq: xibound}
|\xi^{n+1}|=\frac{r^{n+1}}{E(\bar{u}^{n+1})} \le \frac{2M}{(\mathcal{L}\bar{u}^{n+1}, \bar{u}^{n+1})+2}.
\end{equation}
Let $\eta_k^{n+1}=1-(1-\xi^{n+1})^{k+1}$, we have $\eta_k^{n+1}=\xi^{n+1}P_{k}(\xi^{n+1})$ with   $P_{k}$ being a polynomial  of degree $k$. Then, we derive from \eqref{eq: xibound} that  there exists $M_k>0$ such that
\begin{equation*}\label{eq: etabound}
|\eta_k^{n+1}|= |\xi^{n+1}P_{k}(\xi^{n+1})| \le \frac{{M_k}}{(\mathcal{L}\bar u^{n+1}, \bar u^{n+1})+2},
\end{equation*}
which, along with  $u^{n+1}=\eta_k^{n+1}\bar u^{n+1}$, implies
\begin{equation*}
\begin{split}
(\mathcal{L}u^{n+1}, u^{n+1})&=(\eta_k^{n+1})^2(\mathcal{L}\bar{u}^{n+1}, \bar{u}^{n+1}) \\
&\le \big(\frac{{M_k}}{(\mathcal{L}\bar u^{n+1}, \bar u^{n+1})+2}\big)^2(\mathcal{L}\bar{u}^{n+1}, \bar{u}^{n+1}) \le M^2_k.
\end{split}
\end{equation*}
The proof is complete.
\end{proof}
\begin{rem}
From the above proof, we observe that it is essential to introduce $\bar u^{n+1}$ and $\eta_k^{n+1}$ in order to obtain \eqref{eq: bound}, and that the bound constant $M_k$ increases as $k$ increases. So while we can replace $k+1$ in $\eta_k^{n+1}$  by any larger integer without affecting the $k$th order accuracy, it is best to use the smallest possible integer, which is $k+1$ for $k$th order accuracy.
\end{rem}

\begin{rem}
 Note that the proof of \eqref{eq: bound} does not depend on specific form of  \eqref{eq: Nsav1}, so the result is also valid if we replace  \eqref{eq: Nsav1} in the scheme by other implicit-explicit  multistep schemes.
\end{rem}




\subsection{Numerical examples}
Before we start the error analysis,  we provide numerical examples to validate the convergence rates and demonstrate the advantage of our approach with the usual IMEX scheme. 
\textit{Example 1.} Consider the
Allen-Cahn equation
\begin{equation}\label{eq: AC1}
\frac{\partial u}{\partial t}=\alpha \Delta u- (1-u^2) u+f,
\end{equation}
 and the Cahn-Hilliard equation
 \begin{equation}\label{eq: CH1}
\frac{\partial u}{\partial t}=-m_0\Delta (\alpha \Delta u-(1-u^2)u)+f,
\end{equation}
in $\Omega=(0,2)\times (0,2)$ with periodic boundary condition, and $f$ is chosen such that
the  exact solution is
\begin{equation}
u(x,y,t)=\exp\big(\sin(\pi x)\sin(\pi y)\big)\sin(t).
\end{equation}
We set $\alpha=0.01^2$ in \eqref{eq: AC1} and $\alpha=0.04$, $m_0=0.005$ in \eqref{eq: CH1}, and use the Fourier spectral method with $64\times 64$ modes for space discretization so that the spatial discretization error is negligible when compared with the time discretization error.
 In Figures \ref{fig: ACtest1} (resp. \ref{fig: CHtest1}), we  plot the  convergence rate of the $H^2$ error at $T=1$  for the Allen-Cahn (resp. Cahn-Hilliard) equation. We    observe  the expected convergence rates  for all  cases.

\begin{figure}[htbp]
\begin{center}
  \subfigure[BDF$1$ and BDF$2$ vs errors of $u$ ]{ \includegraphics[scale=.35]{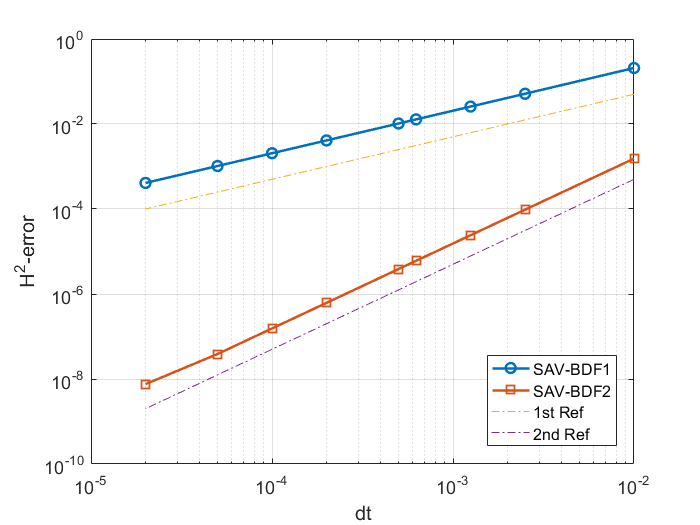}}
   \subfigure[BDF$3,4,5 $ vs errors of $u$]{ \includegraphics[scale=.35]{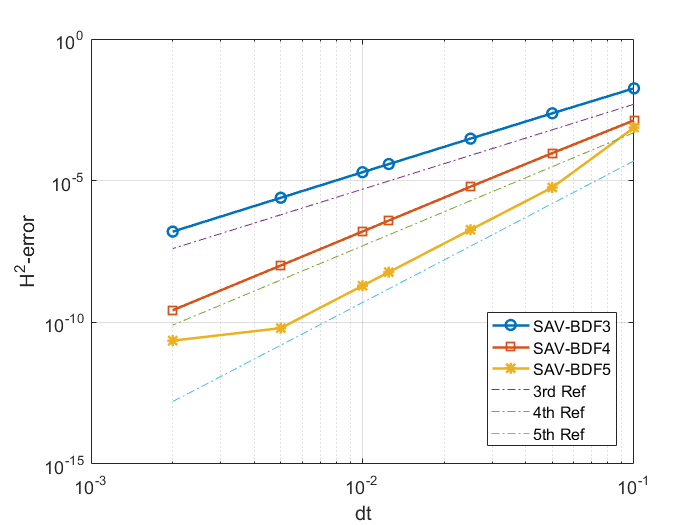}}\\
 \caption{\small   Convergence test for the Allen-Cahn equation using the new SAV/BDF$k$ $(k=1,2,3,4,5)$. (a)-(b) $H^2$ errors of $u$ as a function of $\Delta{t}$.}
   \label{fig: ACtest1}
\end{center}
\end{figure}
\begin{figure}[htbp]
\begin{center}
  \subfigure[BDF$1$ and BDF$2$ vs errors of $u$ ]{ \includegraphics[scale=.35]{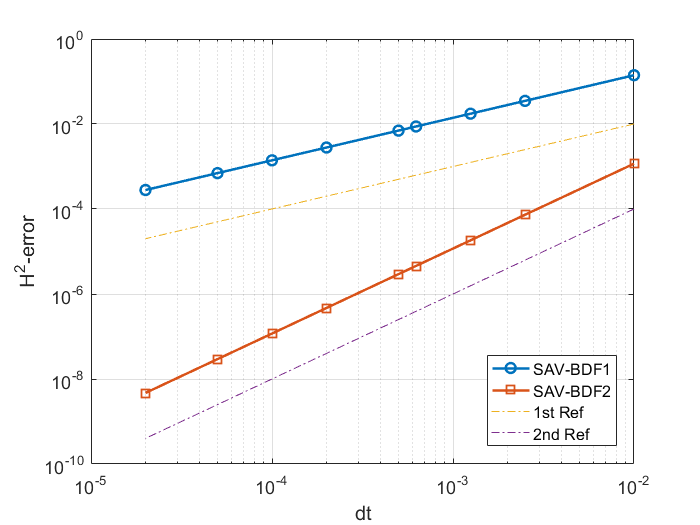}}
   \subfigure[BDF$3,4,5 $ vs errors of $u$]{ \includegraphics[scale=.35]{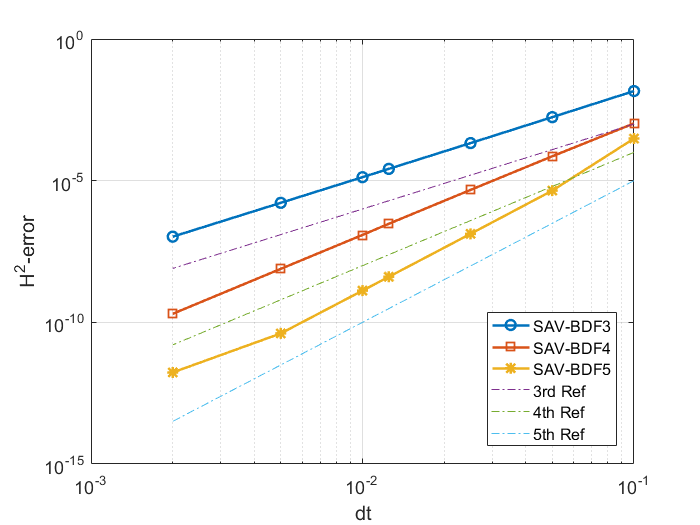}}\\
 \caption{\small  Convergence test for the Cahn-Hilliard equation using the new SAV/BDF$k$ $(k=1,2,3,4,5)$. (a)-(b) $H^2$ errors of $u$ as a function of $\Delta{t}$.}
   \label{fig: CHtest1}
\end{center}
\end{figure}

\textit{Example 2.} Next, we consider  the 1-D Burgers equation 
\begin{equation}\label{eq: BE}
\frac{\partial u}{\partial t}-\nu u_{xx}+  u u_x=0,
\end{equation}
in $\Omega=(-1,1)$ with the initial condition and Dirichlet boundary condition given as
\begin{equation}
u(x,0)=-\sin(\pi x), \quad u(\pm 1,t)=0.
\end{equation}
In this test, we use the second order SAV scheme and the corresponding second-order IMEX scheme with
  $\nu=\frac{1}{314}$, $N=320$, $\delta t=8.5\times 10^{-3}$. The numerical solutions at $T=1$ are plotted in Fig \ref{fig: BEtest2} (a)  solution obtained by the usual IMEX scheme and  (b) solution obtained by the SAV scheme. We observe that the usual IMEX scheme produces oscillatory solutions while the SAV scheme produces the correct solution which is indistinguishable with the reference solution obtained with  $\delta t= 10^{-4}$ in \ref{fig: BEtest2} (c). We also plot in \ref{fig: BEtest2} (d) the SAV factor $\eta^n=1-(1-\xi^n)^3$. We observe that when the solution exhibits large gradients (for $t\in (0.5,1)$), the  SAV factor $\eta^n$ deviates slightly from 1 so that the SAV scheme still produces correct result while the corresponding IMEX scheme produces incorrect result.
  
\begin{figure}[htbp]
\begin{center}
  \subfigure[ usual IMEX scheme ]{ \includegraphics[scale=.35]{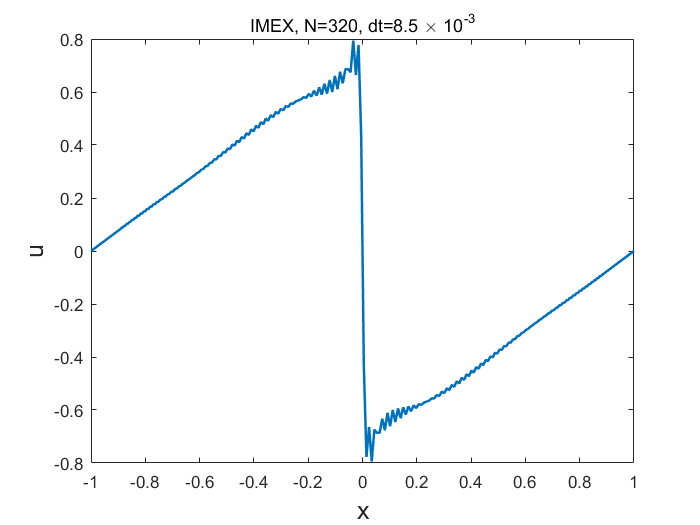}}
   \subfigure[SAV scheme]{ \includegraphics[scale=.35]{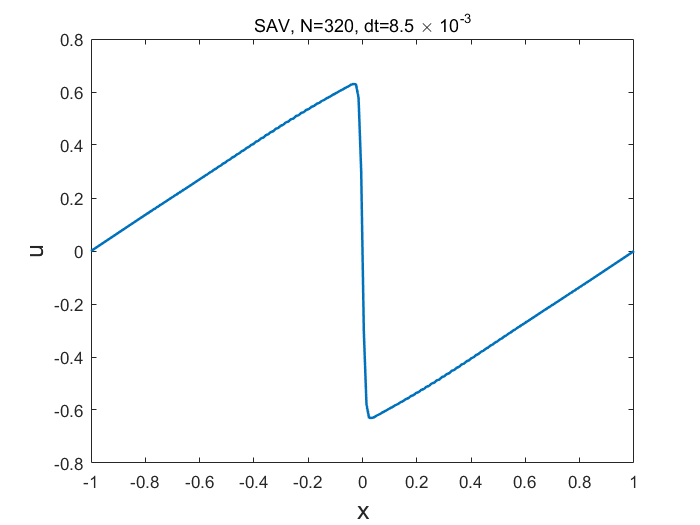}}\\
    \subfigure[reference solution]{ \includegraphics[scale=.35]{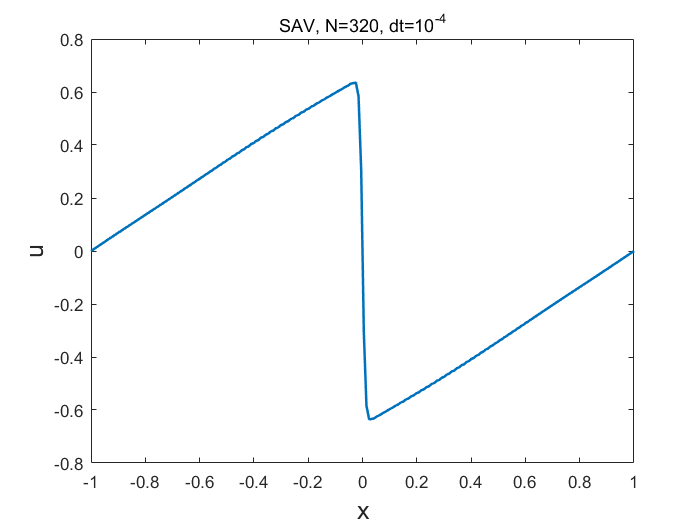}}
     \subfigure[ SAV factor ]{ \includegraphics[scale=.35]{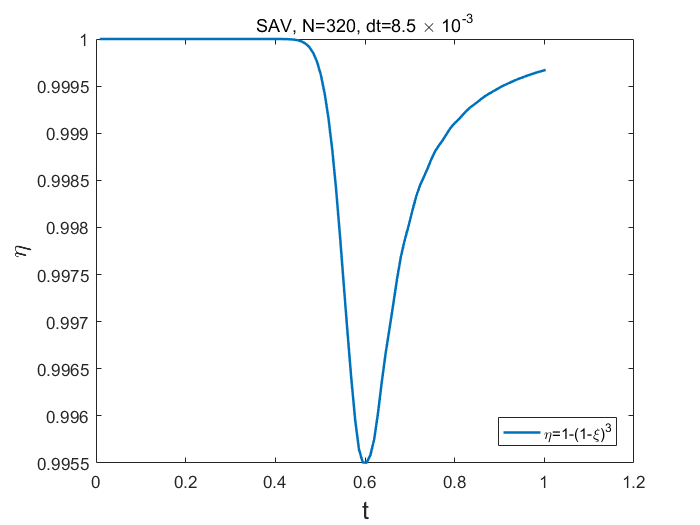}}
  \\
 \caption{\small   Burgers equation: a comparison of usual IMEX and SAV}
   \label{fig: BEtest2}
\end{center}
\end{figure}

\section{Error analysis for Allen-Cahn type equations}
While the  stability results in Theorem \ref{stableThm} are valid for general dissipative systems, it is cumbersome to carry out error analysis with such generality. So to simplify the presentation, we shall carry out error analysis for two class of typical dissipative equations: Allen-Cahn type equation in this section and Cahn-Hilliard type equation in the next section.

We first recall the following important result. Based on Dahlquist's G-stability theory,
Nevanlinna and Odeh  \cite{nevanlinna1981} proved the following results for BDF$k\;(1 \le k\le 5)$ schemes.

\begin{lemma}\label{lemmaODEH}
 For $1\le k\le 5$, there exist $0 \le \tau_k < 1$, a positive definite symmetric matrix $G=(g_{ij}) \in \mathcal{R}^{k,k}$ and real numbers $\delta_0,..., \delta_k$ such that
\begin{equation*}
\begin{split}
\Big(\alpha_k u^{n+1}-A_k(u^n), u^{n+1}-\tau_k u^n \Big)&=\sum_{i,j=1}^{k}g_{ij}(u^{n+1+i-k},u^{n+1+j-k})\\
&-\sum_{i,j=1}^{k}g_{ij}(u^{n+i-k},u^{n+j-k})
+\|\sum_{i=0}^{k} \delta_i u^{n+1+i-k}\|^2,
\end{split}
\end{equation*}
where the smallest possible values of $\tau_k$ are
\begin{equation*}
\tau_1=\tau_2=0,\quad \tau_3=0.0836,\quad \tau_4=0.2878,\quad \tau_5=0.8160,
\end{equation*}
and $\alpha_k$, $A_k$ are defined in \eqref{eq:bdf3}-\eqref{eq:bdf5}.
\end{lemma}
The above result played a key role in proving the stability of high-order BDF schemes for nonlinear parabolic equations \cite{akrivis2015}, and it plays an important role in our error analysis.

We shall also frequently use the following discrete Gronwall Lemma (see for example, \cite{STWbook}, Lemma B.10).
\begin{lemma}
\textbf{(Discrete Gronwall Lemma)} Let $y^k,\,h^k,\,g^k,\,f^k$ be four nonnegative sequences satisfying
\begin{equation*}
y^n+\delta t \sum_{k=0}^{n}h^k \le B+\delta t \sum_{k=0}^{n}(g^ky^k+f^k)\; \text{ with }\; \delta t \sum_{k=0}^{T/\delta t} g^k \le M,\, \forall\, 0\le n \le T/\delta t.
\end{equation*}
We assume $\delta t\, g^k <1$ and let $\sigma=\max_{0\le k \le T/\delta t}(1-\delta t g^k)^{-1}$. Then
\begin{equation*}
y^n+\delta t \sum_{k=1}^{n}h^k \le \exp(\sigma M)(B+\delta t\sum_{k=0}^{n}f^k),\,\,\forall\, n\le T/\delta t.
\end{equation*}
\end{lemma}

Consider the Allen-Cahn type  equation:
\begin{equation}\label{eq: AC}
\frac{\partial u}{\partial t}-\Delta u+\lambda u-g(u)=0 \quad (\bm x,t)\in \Omega\times (0,T],
\end{equation}
where $\Omega$ is an open bounded domain in $\mathbb{R}^d\;(d=1,2,3)$, with the initial condition $u(x,0)=u^0(x)$, and   boundary condition:
\begin{equation}
periodic, 
\; or \; u|_{\partial\Omega}=0,\; or\;  \frac{\partial u}{\partial \bm n}|_{\partial\Omega}=0.
\end{equation}
The above equation is a special case of \eqref{dissE} with  $\mathcal{A}=-\Delta+\lambda I$, and satisfies the dissipation law \eqref{dissL} with $E(u)=\frac12(\mathcal{L}u,u)+(G(u),1)$ where $(\mathcal{L}u,u)=(\nabla u,\nabla u)+\lambda (u,u)$, $G(u)=\int^u g(v)dv$ and $\mathcal{K}(u)=( \frac{\delta E}{\delta u}, \frac{\delta E}{\delta u})$.
We assume, without loss of generality,
\begin{equation}\label{barc}
 \int_\Omega G(v) dx\ge \underbar{C} >0 \quad\forall v.
\end{equation}
In particular, with $g(u)=(1-u^2)u$ and $\lambda=0$, the above equation becomes the celebrated Allen-Cahn equation \cite{All.C79}.



We recall  the following regularity  result for  \eqref{eq: AC} (see, for instance,   \cite{temam2012}).
\begin{theorem}\label{Thm2.6}
Assume $u^0 \in H^2(\Omega)$ and the following holds
\begin{equation}\label{eq: 2.6}
|g'(x)|< C(|x|^p+1),\quad  p>0 \,\,  arbitrary \quad  \,if \,  \,d=1,2;\quad  \, 0<p<4 \quad if \, d=3.
\end{equation}
Then for any $T>0$, the problem \eqref{eq: AC} has a unique solution in the space
\begin{equation*}
C([0,T];H^2(\Omega)) \cap L^2(0,T; H^3(\Omega)).
\end{equation*}
\end{theorem}
We also recall a  result (see Lemma 2.3 in \cite{SX18}) which is useful to deal with the nonlinear term in \eqref{eq: AC}.
\begin{lemma}\label{lemma}
Assume that $\|u\|_{H^1} \le M$ and \eqref{eq: 2.6} holds. Then for any $u \in H^3$, there exist $0 \le \sigma <1$ and a constant $C(M)$ such that the following inequality holds:
\begin{equation*}
\|\nabla g(u)\|^2 \le C(M)(1+\|\nabla \Delta u\|^{2\sigma}).
\end{equation*}
\end{lemma}

We denote hereafter
$$t^n=n\,\delta t,\,\bar{e}^n=\bar{u}^{n}- u(\cdot,t^n),\, e^n=u^n- u(\cdot,t^n),\, s^n=r^n-r(t^n).$$

In the following, we carry out a unified error analysis for the first- to fifth- order SAV schemes described as in \eqref{eq: Nsav} with the coefficients defined in \eqref{eq:bdf3} - \eqref{eq:bdf5}.

 For \eqref{eq: AC}, the $k$th-order version of
 \eqref{eq: Nsav1} and \eqref{Nsav4} read:
 \begin{equation}\label{eq: ubark}
\frac{\alpha_k \bar{u}^{n+1}-A_k(u^n)}{ \delta t}=\Delta \bar{u}^{n+1}-\lambda \bar{u}^{n+1}+g[B_k(\bar{u}^n)],
\end{equation}
\begin{equation}
\frac{\alpha_k u^{n+1}-\eta^{n+1}_{k}A_k(u^n)}{ \delta t}=\Delta u^{n+1}-\lambda u^{n+1}+\eta_k^{n+1}g[B_k(\bar{u}^n)],\label{eq: Nsavk}
\end{equation}
where $\alpha_k$, $A_k$, $B_k$ defined in \eqref{eq:bdf3} - \eqref{eq:bdf5}.

\begin{theorem}\label{ThmAC} Given initial condition $\bar{u}^0=u^0=u(0)$,  $r^0=E[u^0]$.
Let $\bar u^{n+1}$ and $u^{n+1}$ be computed with the $k$th order scheme  \eqref{eq: Nsav1}- \eqref{eq: Nsav2} $(1 \le k \le5)$ for \eqref{eq: AC} with
\begin{equation*}
 \eta_1^{n+1}=1-(1-\xi^{n+1})^{3}, \quad \eta_k^{n+1}=1-(1-\xi^{n+1})^{k+1}\; (k=2,3,4,5).
\end{equation*}
We assume \eqref{eq: 2.6} holds and
\begin{equation*}
u^0\, \in H^3,\; \frac{\partial ^j u}{\partial t^j} \in L^2(0,T;H^1) \,\, 1\le j \le k+1.
\end{equation*}
Then for $n+1 \le T/\delta t$ and $\delta t< \min\{\frac{1}{1+2C_0^{k+2}},\frac{1-\tau_k}{3k}\}$, we have
\begin{equation*}
\|\bar e^{n+1}\|_{H^2},\; \|e^{n+1}\|_{H^2}\le C \delta t^{k},
\end{equation*}
where the constants $C_0$, $C$ are dependent on $T,\, \Omega,$ the $k\times k$ matrix $G=(g_{ij})$ in Lemma \ref{lemmaODEH} and the exact solution $u$ but are independent of $\delta t$ and $0 \le \tau_k<1 $ is the constant in Lemma \ref{lemmaODEH}.
\end{theorem}
\begin{proof} We assume that $\bar u^i$ and $u^i$ $(i=1,\cdots,k-1)$ are computed with a proper initialization procedure such that $\|\bar u^i-u(t_i)\|_{H^2}=O(\delta t^k)$ and $\| u^i-u(t_i)\|_{H^2}=O(\delta t^k)$ $(i=1,\cdots,k-1)$. To simplify the presentation, we set
 $\bar{u}^i=u^i=u(t_i)$ and $r^i=E_1[u^i]$ for $i=1,\cdots,k-1$.

The main task  is to prove
\begin{equation}\label{eq: prestepsk}
|1-\xi^q| \le C_0\,\delta t,\,\, \forall  q\le T/{\delta t}.
\end{equation}
where  the constant $C_0$ is dependent on $T,\, \Omega$ and the exact solution $u$ but is independent of $\delta t$, and  will be defined in the proof process. Below we shall prove
\eqref{eq: prestepsk} by induction.

Under the assumption, \eqref{eq: prestepsk}  certainly holds for $q=0$. Now
suppose we have
\begin{equation}\label{eq: prestepsk2}
|1-\xi^q| \le C_0\,\delta t,\,\, \forall q \le m,
\end{equation}
we shall prove below
\begin{equation}\label{eq: xik}
|1-\xi^{m+1}| \le C_0 \delta t.
\end{equation}
We shall first consider $k=2,3,4,5$, and point out the necessary modifications for the case $k=1$ later.

\textbf{Step 1: $H^2$ bound for $u^{n}$ and $\bar{u}^{n}$ for all $n\le m$.} For the $k$th-order schemes, it follows from Theorem \ref{stableThm} that
\begin{equation}\label{eq: H1boundHigh}
\|u^q\|_{H^1} \le M_k,\,\forall q\le T/{\delta t}.
\end{equation}
Under assumption \eqref{eq: prestepsk2}, if we choose $\delta t$ small enough such that
\begin{equation}\label{eq: dtcond1}
 \delta t \le \text{min}\{\frac{1}{2C_0^{k+1}},1\},
\end{equation}
 we have
\begin{equation}\label{eq: erroretaHigh}
1-\frac{\delta t^k}{2} \le |\eta_k^q| \le 1+ \frac{\delta t^k}{2},\,\,|1-\eta_k^q| \le \frac{\delta t^k}{2}, \forall q\le m,
\end{equation}
and
\begin{equation}\label{eq: ubarH1boundHigh}
\|\bar{u}^q\|_{H^1} \le 2 M_k,\,\forall q\le m,\,\forall \delta t\le 1.
\end{equation}
Consider \eqref{eq: Nsavk} in step $q$:
\begin{equation}\label{eq: Nsavkn}
\frac{\alpha_k u^{q}-\eta^{q}_{k}A_k(u^{q-1})}{ \delta t}=\Delta u^{q}-\lambda u^{q}+\eta_k^{q}g[B_k(\bar{u}^{q-1})].
\end{equation}
Thanks to Lemma \ref{lemma} and \eqref{eq: erroretaHigh}, we have
\begin{equation*}
\begin{split}
\|\nabla g[B_k(\bar{u}^{q-1})]\|^2 & \le C(M_k)(\|\nabla \Delta B_k(\bar{u}^{q-1})\|^{2\sigma}+1)\\
& \le \gamma_k \|\nabla \Delta B_k(\bar{u}^{q-1})\|^{2}+C(M_k,\gamma_k)\\
&  \le \gamma_k \|\nabla \Delta B_k(\frac{1}{\eta_k^{q-1}}{u}^{q-1})\|^{2}+C(M_k,\gamma_k)\\
& \le 4 \gamma_k \sum_{i=1}^{k} \| \nabla \Delta u^{q-i}\|^2+C(M_k,\gamma_k),
\end{split}
\end{equation*}
where $\gamma_k$ can be any positive constant.
Taking the inner product of  \eqref{eq: Nsavkn} with $\Delta^2 {u}^{q}-\tau_k \Delta^2 {u}^{q-1}$ and using the above inequality, it follows from Lemma \ref{lemmaODEH} that there exist $0 \le \tau_k < 1$, a positive definite symmetric matrix $G=(g_{ij}) \in \mathcal{R}^{k,k}$ and  $\delta_0,..., \delta_k$  that
\begin{small}
\begin{equation}\label{eq: 2.15}
\begin{split}
\frac{1}{\delta t} & \Big (\sum_{i,j=1}^{k}g_{ij}(\Delta{u}^{q+i-k},\Delta {u}^{q+j-k}) -\sum_{i,j=1}^{k}g_{ij}(\Delta {u}^{q-1+i-k},\Delta {u}^{q-1+j-k})+\|\sum_{i=0}^{k} \delta_i \Delta {u}^{q+i-k}\|^2 \Big)\\
& +\frac{1}{2}\|\nabla \Delta {u}^{q}\|^2+\frac{\lambda}{2}\|\Delta {u}^{q}\|^2 \\
& \le \eta_k^q \big ( g[B_k(\bar{u}^{q-1})], \Delta ^2 {u}^{q}-\tau_k\Delta^2 {u}^{q-1} \big )+\frac{\tau_k}{2}\|\nabla \Delta {u}^{q-1}\|^2+\frac{\lambda \tau_k}{2}\|\Delta {u}^{q-1}\|^2\\
& +\frac{(\eta_k^q-1)}{\delta t}\big(A_k(u^{q-1}),\Delta^2 {u}^{q}-\tau_k \Delta^2 {u}^{q-1} \big) \\
& \le C(\epsilon_k)|\eta_k^q| \|\nabla g[B_k(\bar{u}^{q-1})]\|^2 +\epsilon_k|\eta_k^q| \big(\|\nabla \Delta {u}^{q}\|^2+\|\nabla \Delta {u}^{q-1}\|^2 \big)+\frac{\tau_k}{2}\|\nabla \Delta {u}^{q-1}\|^2\\
&+\frac{\lambda \tau_k}{2}\|\Delta {u}^{q-1}\|^2
 +\frac{|1-\eta_k^q|}{\delta t}\|\nabla A_k(u^{q-1})\|^2+ \frac{|1-\eta_k^q|}{\delta t} \big(\|\nabla \Delta u^q\|^2+ \|\nabla \Delta u^{q-1}\|^2\big)\\
& \le C(M_k,\,\epsilon_k,\,\gamma_k)+ \big(4C(\epsilon_k)|\eta_k^q|\gamma_k+\epsilon_k|\eta_k^q|+\frac{|1-\eta_k^q|}{\delta t}\big)\sum_{i=1}^{k}\|\nabla \Delta {u}^{q-i}\|^2 \\
& +\frac{\tau_k}{2}\|\nabla \Delta {u}^{q-1}\|^2+\frac{\lambda \tau_k}{2}\|\Delta {u}^{q-1}\|^2+\frac{|1-\eta_k^q|}{\delta t}\|\nabla A_k(u^{q-1})\|^2,
\end{split}
\end{equation}
\end{small}
where $\epsilon_k$ can be any positive constant.
Note that $\tau_k<1$, we can choose $\delta t$,  $\epsilon_k$ and $\gamma_k$  small enough such that
\begin{equation}\label{cond5}
\delta t < \frac{1-\tau_k}{3k},\quad \epsilon_k<\frac{1-\tau_k}{12k},\quad \gamma_k<\frac{1-\tau_k}{48k C(\epsilon_k)},
\end{equation}
with the estimate in \eqref{eq: erroretaHigh}, we have
\begin{equation}
\begin{split}
4 C(\epsilon_k) |\eta_k^q|\gamma_k+\epsilon_k|\eta_k^q|+\frac{1-\eta_k^q}{\delta t} & \le 8C(\epsilon_k)\gamma_k+2\epsilon_k+ \frac{\delta t^{k-1}}{2} \\
& \le \frac{1-\tau_k}{6k}+\frac{1-\tau_k}{6k}+\frac{1-\tau_k}{6k}\\
& \le \frac{1-\tau_k}{2k}.
\end{split}
\end{equation}

Then, taking the sum   \eqref{eq: 2.15} for $q$ from $k-1$ to $n \;(\le m)$, we obtain
\begin{equation*}
\begin{split}
 \sum_{i,j=1}^{k}g_{ij}&(\Delta {u}^{n+i-k},\Delta {u}^{n+j-k}) \\
& \le C(M_k,\,\tau_k)T+C(u^0,...,u^{k-1})+C_{A_k}k \delta t^{k}\sum_{q=0}^{n-1}\|\nabla {u}^q\|^2\\
& \le C(M_k,\,\tau_k)T+C(u^0,...,u^{k-1})+C_{A_k} k \delta t^{k-1}T M_k^2,
\end{split}
\end{equation*}
where $C(M_k,\,\tau_k)$ is a constant only depends on $M_k,\,\tau_k$,  $C(u^0,...,u^{k-1})$ only depends on $u^0,...,u^{k-1}$ and $C_{A_k}$ only depends on the coefficients in $A_k$.
Since $G=(g_{ij})$ is a positive definite symmetric matrix, we have
\begin{equation*}
\begin{split}
\lambda_G \|\Delta {u}^{n}\|^2 & \le \sum_{i,j=1}^{k}g_{ij} (\Delta{u}^{n+i-k},\Delta {u}^{n+j-k})\\
& \le C(M_k,\,\tau_k)T+C(u^0,...,u^{k-1})+C_{A_k} k \delta t^{k-1}T M_k^2.
\end{split}
\end{equation*}
where $\lambda_G>0$ is the minimum eigenvalue of $G=(g_{ij})$.
Together with \eqref{eq: H1boundHigh}, the above  implies
\begin{equation}\label{eq: uH2boundk}
\|{u}^{n}\|_{H^2}\le \frac{1}{\lambda_G}\sqrt{C(M_k,\,\tau_k)T+C(u^0,...,u^{k-1})+C_{A_k} k  T M_k^2}:=C_1,\, \forall \delta t<1,\; n\le m.
\end{equation}
Noting that
\begin{equation*}
\|{u}^{n}\|_{H^2}=|\eta_k^n|\|\bar{u}^n\|_{H_2},
\end{equation*}
then \eqref{eq: erroretaHigh} implies
\begin{equation}\label{eq: baruH2boundk}
\|\bar{u}^n\|_{H^2}\le 2 C_1, \,\,\forall \delta t<1,\; n\le m.
\end{equation}

\textbf{Step 2: estimate for $\|\bar{e}^{n+1}\|_{H^2}$ for all $0\le n\le m$.}
By Theorem \ref{Thm2.6} and \eqref{eq: baruH2boundk}
 we can choose  $C$ large enough such that
\begin{equation}\label{eq: H2high}
\|u(t)\|_{H^2}\le C,\,\, \forall t \le T,\,\|\bar{u}^q\|_{H^2} \le C,\,\forall q\le m.
\end{equation}
Since $H^2 \subset L^{\infty}$, without loss of generality, we can adjust $C$ such that
\begin{equation}\label{eq: ghigh}
|g^{(i)}[u(t)]|_{L^\infty}\le C, \,\, \forall t \le T;\,|g^{(i)}(\bar{u}^q)|_{L^\infty} \le C, \,\, \forall q \le m,\, i=0,1,2.
\end{equation}
From \eqref{eq: ubark}, we can write down the error equation  as
\begin{equation}\label{eq: errorebarhigh}
\alpha_k\bar{e}^{n+1}-A_k(\bar{e}^n)=A_k(u^n)-A_k(\bar{u}^n)+\delta t \Delta \bar{e}^{n+1}-\delta t\lambda \bar{e}^{n+1}+ R^n_k +\delta t Q^n_k,
\end{equation}
where $R^n_k$,\, $Q^n_k$ are given by
\begin{equation}\label{eq: R1high}
\begin{split}
R^n_k & =-\alpha_k u(t^{n+1})+A_k(u(t^n))+\delta t u_t(t^{n+1}) \\
& =\sum_{i=1}^{k} a_i \int_{t^{n+1-i}}^{t^{n+1}}(t^{n+1-i}-s)^k\frac{\partial^{k+1} u}{\partial t^{k+1}}(s)ds,
\end{split}
\end{equation}
with $a_i$ being some fixed and bounded constants determined by the truncation errors, and
\begin{equation}\label{eq: R2high}
Q^n_k=g[B_k(\bar{u}^n)]-g[u(t^{n+1})].
\end{equation}
 For example, in the case $k=3$, we have
\begin{equation*}
R_3^n=-3\int_{t^n}^{t^{n+1}}(t^n-s)^3 \frac{\partial^{4} u}{\partial t^{4}}(s)ds+\frac{3}{2}\int_{t^{n-1}}^{t^{n+1}}(t^{n-1}-s)^3 \frac{\partial^{4} u}{\partial t^{4}}(s)ds-\frac{1}{3}\int_{t^{n-2}}^{t^{n+1}}(t^{n-2}-s)^3 \frac{\partial^{4} u}{\partial t^{4}}(s)ds.
\end{equation*}

Taking the inner product of  \eqref{eq: errorebarhigh} with $\bar{e}^{n+1}-\tau_k \bar{e}^{n}$, it follows from Lemma \ref{lemmaODEH} that
\begin{equation}\label{eq: errorhigh}
\begin{split}
\sum_{i,j=1}^{k}g_{ij}&(\bar{e}^{n+1+i-k}, \bar{e}^{n+1+j-k}) -\sum_{i,j=1}^{k}g_{ij}( \bar{e}^{n+i-k}, \bar{e}^{n+j-k})\\
&+\|\sum_{i=0}^{k} \delta_i \bar{e}^{n+1+i-k}\|^2
+\delta t \|\nabla \bar{e}^{n+1}\|^2 +\lambda \delta t\|\bar{e}^{n+1}\|^2 \\
& = (A_k({u}^n)-A_k(\bar{u}^n), \bar{e}^{n+1}-\tau_k \bar{e}^{n})-\delta t (\Delta \bar{e}^{n+1}, \tau_k \bar{e}^{n})+\delta t \lambda(\bar{e}^{n+1},\tau_k \bar{e}^n)\\
&+ (R^n_k, \bar{e}^{n+1}-\tau_k \bar{e}^{n})
 +\delta t (Q^n_k, \bar{e}^{n+1}-\tau_k \bar{e}^{n}).
\end{split}
\end{equation}
In the following, we bound the right hand side of \eqref{eq: errorhigh}. Note that
\begin{equation*}
u^q=\eta_k^q\bar{u}^q, \quad |\eta_k^q-1| \le C_0^{k+1}\, \delta t^{k+1},\quad \forall q\le n.
\end{equation*}
 hence
\begin{equation}\label{eq: erroretahigh}
\begin{split}
|(A_k({u}^n)-A_k(\bar{u}^n), \bar{e}^{n+1}-\tau_k \bar{e}^{n})| & \le \frac{\|A_k({u}^n)-A_k(\bar{u}^n)\|^2}{2\delta t} +\frac{\delta t}{2}\|\bar{e}^{n+1}-\tau_k\bar{e}^n\|^2 \\
& \le CC_0^{2k+2} \delta t^{2k+1} +\delta t \|\bar{e}^{n+1}\|^2+\delta t \|\bar{e}^{n}\|^2.
\end{split}
\end{equation}
It follows from \eqref{eq: R1high} that
\begin{equation}\label{eq: absR1high}
\|R^n_k\|^2 \le C \delta t^{2k+1} \int_{t^{n+1-k}}^{t^{n+1}} \|\frac{\partial^{k+1}u}{\partial t^{k+1}}(s)\|^2 ds.
\end{equation}
And we can bound $Q^n_k$ based on \eqref{eq: ghigh} and \eqref{eq: R2high} as
\begin{equation}\label{eq: absR2high}
\begin{split}
|Q^n_k|& =\big |g[B_k(\bar{u}^n)]-g[B_k(u(t^n))]+g[B_k(u(t^n))]-g[u(t^{n+1})] \big | \\
 & \le C |B_k(\bar{e}^n)| +C |B_k(u(t^n))-u(t^{n+1})| \\
 & = C |B_k(\bar{e}^n)| +C \Big|\sum_{i=1}^{k}b_i \int_{t^{n+1-i}}^{t^{n+1}}(t^{n+1-i}-s)^{k-1}\frac{\partial^{k} u}{\partial t^{k}}(s)ds \Big |,
\end{split}
\end{equation}
where $b_i$ are some fixed and bounded constants determined by the truncation error. For example, in the case $k=3$, we have
\begin{equation*}
\begin{split}
B_3(u(t^n))-u(t^{n+1})& =-\frac{3}{2}\int_{t^{n}}^{t^{n+1}}(t^{n}-s)^2\frac{\partial^{3} u}{\partial t^{3}}(s)ds+\frac{3}{2}\int_{t^{n-1}}^{t^{n+1}}(t^{n-1}-s)^2\frac{\partial^{3} u}{\partial t^{3}}ds \\
& -\frac{1}{2}\int_{t^{n-2}}^{t^{n+1}}(t^{n-2}-s)^2\frac{\partial^{3} u}{\partial t^{3}}ds.
\end{split}
\end{equation*}
Therefore,
\begin{equation}\label{eq: errorR1high}
\begin{split}
|\big(R^n_k, \bar{e}^{n+1}-\tau_k \bar{e}^{n}\big)| & \le \frac{1}{2\delta t}\|R^n_k\|^2+ \delta t \|\bar{e}^{n+1}\|^2+\delta t \|\bar{e}^{n}\|^2 , \\
& \le  \delta t \|\bar{e}^{n+1}\|^2+\delta t \|\bar{e}^{n}\|^2 +C \delta t^{2k} \int_{t^{n+1-k}}^{t^{n+1}} \|\frac{\partial^{k+1}u}{\partial t^{k+1}}(s)\|^2 ds.
\end{split}
\end{equation}
\begin{equation}\label{eq: errorR2high}
\delta t |\big (Q^n_k, \bar{e}^{n+1}- \tau_k \bar{e}^{n} \big)| \le C\delta t \big (\|B_k(\bar{e}^n)\|^2+ \|\bar{e}^{n+1}\|^2+\|\bar{e}^{n}\|^2 \big) + C \delta t^{2k} \int_{t^{n+1-k}}^{t^{n+1}}\|\frac{\partial^{k}u}{\partial t^{k}}(s)\|^2ds.
\end{equation}
Now, combining \eqref{eq: errorhigh}, \eqref{eq: erroretahigh}, \eqref{eq: errorR1high}, \eqref{eq: errorR2high}, we arrive at
\begin{equation*}
\begin{split}
\sum_{i,j=1}^{k}g_{ij}&(\bar{e}^{n+1+i-k}, \bar{e}^{n+1+j-k}) -\sum_{i,j=1}^{k}g_{ij}( \bar{e}^{n+i-k}, \bar{e}^{n+j-k})\\
&+\|\sum_{i=0}^{k} \delta_i \bar{e}^{n+1+i-k}\|^2
+\frac{1}{2}\delta t \|\nabla \bar{e}^{n+1}\|^2
+ \frac{\lambda}{2}\delta t\|\bar{e}^{n+1}\|^2\\
& \le \frac{\tau_k}{2} \delta t \|\nabla \bar{e}^{n}\|^2+\frac{\lambda\tau_k}{2}\delta t\|\bar{e}^n\|^2+ C C_0^{2k+2} \delta t^{2k+1} + C\delta t \sum_{i=0}^{k}\|\bar{e}^{n+1-i}\|^2\\
& +C \delta t^{2k} \int_{t^{n+1-k}}^{t^{n+1}}(\|\frac{\partial^{k}u}{\partial t^{k}}(s)\|^2+\|\frac{\partial^{k+1}u}{\partial t^{k+1}}(s)\|^2)ds.
\end{split}
\end{equation*}
Taking the sum of the above for $n$ from $k-1$ to $m$,  noting that $G=(g_{ij})$ is a positive definite symmetric matrix with minimum eigenvalue $\lambda_G$, we obtain:
\begin{equation}\label{eq: 2.33}
\begin{split}
 \lambda_G \|\bar{e}^{m+1}\|^2 & \le  \sum_{i,j=1}^{k}g_{ij} (\bar{e}^{m+1+i-k}, \bar{e}^{m+1+j-k}) \\
& \le C \delta t \sum_{q=0}^{m+1} \|\bar{e}^q\|^2 +C \delta t^{2k} \int_{0}^{T} (\|\frac{\partial^{k}u}{\partial t^{k}}(s)\|^2+\|\frac{\partial^{k+1}u}{\partial t^{k+1}}(s)\|^2+C_0^{2k+2})ds
\end{split}
\end{equation}
We can obtain similar inequalities for $\|\nabla \bar{e}^m\|$ and $\|\Delta \bar{e}^m\|$ by using essentially the same procedure. Indeed,
 taking the inner product  of \eqref{eq: errorebarhigh} with $-\Delta \bar{e}^{n+1}+\tau_k \Delta \bar{e}^n$, by using Lemma \ref{lemmaODEH}, we obtain
 \begin{equation}\label{eq: errorhigh2}
\begin{split}
& \sum_{i,j=1}^{k}g_{ij}(\nabla \bar{e}^{n+1+i-k}, \nabla \bar{e}^{n+1+j-k}) -\sum_{i,j=1}^{k}g_{ij}(\nabla  \bar{e}^{n+i-k}, \nabla \bar{e}^{n+j-k})+\|\sum_{i=0}^{k} \delta_i \nabla \bar{e}^{n+1+i-k}\|^2\\
&+\delta t \|\Delta \bar{e}^{n+1}\|^2 +\lambda \delta t\|\nabla \bar{e}^{n+1}\|^2 \\
& = (\nabla A_k({u}^n)-\nabla A_k(\bar{u}^n), \nabla \bar{e}^{n+1}-\tau_k \nabla \bar{e}^{n})+\delta t (\Delta \bar{e}^{n+1}, \tau_k \Delta \bar{e}^{n})+\delta t \lambda(\nabla \bar{e}^{n+1},\tau_k \nabla \bar{e}^n)\\
&+ (R^n_k, -\Delta \bar{e}^{n+1}+\tau_k \Delta \bar{e}^{n})  +\delta t (Q^n_k, -\Delta \bar{e}^{n+1}+\tau_k \Delta \bar{e}^{n}).
\end{split}
\end{equation}
Taking the sum of the above for $n$ from $k-1$ to $m$, using Lemma \ref{lemmaODEH}, \eqref{eq: absR1high} and \eqref{eq: absR2high}, we can obtain
\begin{equation}\label{eq: H1boundhigh}
\begin{split}
 \lambda_G \|\nabla \bar{e}^{m+1}\|^2 & \le  \sum_{i,j=1}^{k}g_{ij} (\nabla \bar{e}^{m+1+i-k}, \nabla \bar{e}^{m+1+j-k}) \\
& \le C \delta t \sum_{q=0}^{m+1} \|\nabla \bar{e}^q\|^2 +C \delta t^{2k} \int_{0}^{T} (\|\frac{\partial^{k}u}{\partial t^{k}}(s)\|^2+\|\frac{\partial^{k+1}u}{\partial t^{k+1}}(s)\|^2+C_0^{2k+2})ds.
\end{split}
\end{equation}
On the other hand, taking the inner product of \eqref{eq: errorebarhigh} with $\Delta^2 \bar{e}^{n+1}-\tau_k \Delta^2 \bar{e}^n$, by using Lemma \ref{lemmaODEH}, we obtain
\begin{equation}\label{eq: errorhigh3}
\begin{split}
& \sum_{i,j=1}^{k}g_{ij}(\Delta \bar{e}^{n+1+i-k}, \Delta \bar{e}^{n+1+j-k}) -\sum_{i,j=1}^{k}g_{ij}( \Delta \bar{e}^{n+i-k}, \Delta \bar{e}^{n+j-k})+\|\sum_{i=0}^{k} \delta_i \Delta \bar{e}^{n+1+i-k}\|^2\\
& +\delta t \|\nabla \Delta \bar{e}^{n+1}\|^2 +\lambda \delta t\|\Delta \bar{e}^{n+1}\|^2 \\
& = (\Delta A_k({u}^n)-\Delta A_k(\bar{u}^n), \Delta \bar{e}^{n+1}-\tau_k \Delta \bar{e}^{n})+\delta t (\nabla \Delta \bar{e}^{n+1}, \tau_k \nabla \Delta \bar{e}^{n})+\delta t \lambda(\Delta \bar{e}^{n+1},\tau_k \Delta \bar{e}^n)\\& + (\nabla R^n_k, -\nabla \Delta \bar{e}^{n+1}+\tau_k \nabla \Delta \bar{e}^{n})
 +\delta t (\nabla Q^n_k, -\nabla \Delta \bar{e}^{n+1}+\tau_k \nabla \Delta \bar{e}^{n}).
\end{split}
\end{equation}
Here, we need to pay attention to the terms with $\nabla \Delta \bar{e}^{n+1}$ or $\nabla \Delta \bar{e}^{n}$. Firstly, we have
\begin{equation*}
|\delta t(\nabla \Delta \bar{e}^{n+1},\tau_k \nabla \Delta \bar{e}^n)|\le \frac{\delta t}{2}\|\nabla \Delta \bar{e}^{n+1}\|^2+\frac{\tau_k^2 \delta t}{2}\|\nabla \Delta \bar{e}^{n}\|^2.
\end{equation*}
It follows from \eqref{eq: R1high} and \eqref{eq: R2high} that
\begin{equation}\label{eq: errorR1high2}
\|\nabla R^n_k\|^2\le C \delta t^{2k+1} \int_{t^{n+1-k}}^{t^{n+1}} \|\nabla \frac{\partial^{k+1}u}{\partial t^{k+1}}(s)\|^2 ds,
\end{equation}
and
\begin{equation}\label{eq: errorR2high2}
\begin{split}
|\nabla Q^n_k|&\le C (|B_k(\bar{e}^n)|+|\nabla B_k(\bar{e}^n)|) +C \Big|\sum_{i=1}^{k}b_i \int_{t^{n+1-i}}^{t^{n+1}}(t^{n+1-i}-s)^{k-1}\frac{\partial^{k} u}{\partial t^{k}}(s)ds \Big |\\
& +C \Big|\sum_{i=1}^{k}b_i \int_{t^{n+1-i}}^{t^{n+1}}(t^{n+1-i}-s)^{k-1}\nabla\frac{\partial^{k} u}{\partial t^{k}}(s)ds \Big |.
\end{split}
\end{equation}
Therefore,
\begin{equation*}
\begin{split}
|(\nabla R^n_k, &-\nabla \Delta \bar{e}^{n+1}+\tau_k \nabla \Delta \bar{e}^{n})|  \le  \frac{C}{\delta t}\|\nabla R^n_k\|^2+ \frac{\delta t(1-\tau_k^2)}{16}\|-\nabla \Delta \bar{e}^{n+1}+\tau_k \nabla \Delta \bar{e}^{n}\|^2\\
& \le C \delta t^{2k}\int_{t^{n+1-k}}^{t^{n+1}}\|\nabla \frac{\partial^{k+1}u}{\partial t^{k+1}}(s)\|^2 ds
+ \frac{\delta t(1-\tau_k^2)}{8}(\|\nabla \Delta \bar{e}^{n+1}\|^2+\|\nabla \Delta \bar{e}^{n}\|^2),
\end{split}
\end{equation*}
and
\begin{equation*}
\begin{split}
\delta t |(\nabla Q^n_k, &-\nabla \Delta \bar{e}^{n+1}+\tau_k \nabla \Delta \bar{e}^{n})|  \le  C \delta t \|\nabla Q^n_k\|^2+ \frac{(1-\tau_k^2)\delta t}{16}\|-\nabla \Delta \bar{e}^{n+1}+\tau_k \nabla \Delta \bar{e}^{n}\|^2\\
& \le C \delta t \|B_k(\bar{e}^n)\|_{H^1}^2 + C \delta t^{2k}\int_{t^{n+1-k}}^{t^{n+1}}\| \frac{\partial^{k}u}{\partial t^{k}}(s)\|_{H^1}^2 ds\\
&+ \frac{(1-\tau_k^2)\delta t}{8}(\|\nabla \Delta \bar{e}^{n+1}\|^2+\|\nabla \Delta \bar{e}^{n}\|^2).
\end{split}
\end{equation*}
We can bound  other terms on the right hand side of \eqref{eq: errorhigh3} as before to arrive at
\begin{equation*}\label{eq: H2high1}
\begin{split}
& \sum_{i,j=1}^{k}g_{ij}(\Delta \bar{e}^{n+1+i-k}, \Delta \bar{e}^{n+1+j-k}) -\sum_{i,j=1}^{k}g_{ij}( \Delta \bar{e}^{n+i-k}, \Delta \bar{e}^{n+j-k})\\
& +\frac{(1+\tau_k^2) \delta t}{4} \|\nabla \Delta \bar{e}^{n+1}\|^2 +\frac{\lambda \delta t}{2}\|\Delta \bar{e}^{n+1}\|^2 \\
& \le C\delta t(\|B_k(\bar{e}^n)\|_{H^1}^2+\|\Delta \bar{e}^{n+1}\|^2+\|\Delta \bar{e}^{n}\|^2)+\frac{(1+\tau_k^2) \delta t}{4}\|\nabla \Delta \bar{e}^n\|^2+\frac{\lambda \tau_k^2 \delta t}{2}\|\Delta \bar{e}^n\|^2\\
&+C \delta t^{2k}\int_{t^{n+1-k}}^{t^{n+1}}(\| \frac{\partial^{k}u}{\partial t^{k}}(s)\|_{H^1}^2+\| \frac{\partial^{k+1}u}{\partial t^{k+1}}(s)\|_{H^1}^2+C_0^{2k+2}) ds.
\end{split}
\end{equation*}
Then, taking the sum of the above for $n$ from $k-1$ to $m$, we obtain
\begin{equation}\label{eq: H2boundhigh}
\begin{split}
\lambda_G \|\Delta \bar{e}^{m+1}\|^2 & \le \sum_{i,j=1}^{k}g_{ij}(\Delta \bar{e}^{m+1+i-k}, \Delta \bar{e}^{m+1+j-k})\\
& \le C\delta t  \sum_{q=0}^{m+1} \|\bar{e}^q\|_{H^2}^2+C \delta t^{2k}\int_{0}^{T}(\| \frac{\partial^{k}u}{\partial t^{k}}(s)\|_{H^1}^2+\| \frac{\partial^{k+1}u}{\partial t^{k+1}}(s)\|_{H^1}^2+C_0^{2k+2}) ds.
\end{split}
\end{equation}
Summing up \eqref{eq: 2.33}, \eqref{eq: H1boundhigh} and \eqref{eq: H2boundhigh}, we obtain
\begin{equation}\label{eq: bareH2}
\lambda_G\|\bar{e}^{m+1}\|_{H^2}^2 \le C\delta t  \sum_{q=0}^{m+1} \|\bar{e}^q\|_{H^2}^2+C \delta t^{2k}\int_{0}^{T}(\| \frac{\partial^{k}u}{\partial t^{k}}(s)\|_{H^1}^2+\| \frac{\partial^{k+1}u}{\partial t^{k+1}}(s)\|_{H^1}^2+C_0^{2k+2}) ds
\end{equation}
Finally, we can obtain the following $H^2$ estimate for $\bar{e}^{m+1}$ by applying the discrete Gronwall lemma to \eqref{eq: bareH2} with $\delta t<\frac{1}{2C}$:
\begin{equation}\label{eq: errorbarehigh}
\begin{split}
\| \bar{e}^{m+1}\|_{H^2}^2 & \le C\exp((1-\delta tC)^{-1}))\delta t^{2k} \int_{0}^{T} (\|\frac{\partial^{k}u}{\partial t^{k}}(s)\|_{H^1}^2+\|\frac{\partial^{k+1}u}{\partial t^{k+1}}(s)\|_{H^1}^2+C_0^{2k+2})ds\\
& \le C_2(1+C_0^{2k+2})\delta t^{2k}\quad \forall 0\le n\le m.
\end{split}
\end{equation}
where $C_2$ is independent of $\delta t$ and $C_0$, can be defined as
\begin{equation}\label{C2high}
C_2:=C \exp(2)\max \big( \int_{0}^{T} (\|\frac{\partial^{k}u}{\partial t^{k}}(s)\|_{H^1}^2+\|\frac{\partial^{k+1}u}{\partial t^{k+1}}(s)\|_{H^1}^2)ds, 1 \big).
\end{equation}
then $\delta t<\frac{1}{2C}$ can be guaranteed by
\begin{equation}\label{cond2}
\delta t <\frac{1}{C_2}.
\end{equation}
In particular, \eqref{eq: errorbarehigh} implies
\begin{equation}\label{eq: eH2high}
\|\bar{e}^{n+1}\|_{H^2}\le \sqrt{C_2(1+C_0^{2k+2})} \delta t^k,\quad \forall 0\le n\le m.
\end{equation}
Combining \eqref{eq: H2high} and \eqref{eq: eH2high}, under the condition \eqref{eq: dtcond1} we obtain
\begin{equation}\label{eq: Cbarhigh}
\|\bar{u}^{n+1}\|_{H^2} \le \sqrt{C_2(1+C_0^{2k+2})}\delta t^2+ C \le \sqrt{C_2(1+1)}+C:=\bar{C} \quad  0\le n\le m.
\end{equation}
Note that $H^2 \subset L^{\infty}$, without loss of generality, we can adjust
$\bar C$ independent of $C_0$ and $\delta t$ so that  we have
\begin{equation}\label{eq: Cbar2high}
\|g(\bar{u}^{n+1})\|,\;\|g'(\bar{u}^{n+1})\|\le \bar C\quad \forall 0\le n\le m.
\end{equation}

\textbf{Step 3: estimate for $|1-\xi^{m+1}|$}. By direct calculation,
\begin{equation}\label{eq: rtt}
r_{tt}=\int_{\Omega}\big( |\nabla u_t| ^2+ \nabla u \cdot \nabla u_{tt}+ \lambda u_{t}^2+\lambda u u_{tt}+g'(u)u_t^2+g(u)u_{tt}\big) d \bm x.
\end{equation}
It follows from \eqref{eq: Nsav3} that the equation for the errors can be written as
\begin{equation}\label{eq: errorxifirst}
s^{n+1}-s^{n}=
 \delta t\big(\|h[u(t^{n+1})]\|^2-\frac{r^{n+1}}{E(\bar{u}^{n+1})}\|h(\bar{u}^{n+1})\|^2 \big) +T_1^n,
\end{equation}
where $h(u)=\frac{\delta E}{\delta u}=-\Delta u +\lambda u- g(u)$, and
\begin{equation}\label{eq: T1}
T_1^n =r(t^n)-r(t^{n+1})+\delta t r_t(t^{n+1})= \int_{t^n}^{t^{n+1}}(s-t^n)r_{tt}(s)ds.
\end{equation}
Taking the sum of \eqref{eq: errorxifirst} for $n$ from $0$ to $m$,  and noting that $s^0=0$, we have
\begin{equation}\label{eq: errorxifirst2}
s^{m+1}=
\delta t\sum_{q=0}^{m}\big(\|h[u(t^{q+1})]\|^2
-\frac{r^{q+1}}{E(\bar{u}^{q+1})}\|h(\bar{u}^{q+1})\|^2 \big) +\sum_{q=0}^{m}T_1^q.
\end{equation}
We can bound the terms on the right hand side of \eqref{eq: errorxifirst2} as follow:
For $T_1^n$, noting \eqref{eq: rtt} we have
\begin{equation}\label{eq: boundT1}
|T_1^n| \le  C \delta t \int_{t^n}^{t^{n+1}}|r_{tt}| ds \le  C \delta t\int_{t^n}^{t^{n+1}}\big(\|u_t(s)\|_{H^1}^2+\|u_{tt}(s)\|_{H^1} \big) ds.
\end{equation}
Next,
\begin{equation}\label{eq: P1P2}
\begin{split}
\big |&\|h[u(t^{n+1})]\|^2-\frac{r^{n+1}}{E(\bar{u}^{n+1})}\|h(\bar{u}^{n+1})\|^2 \big |\\
& \le \|h[u(t^{n+1})]\|^2\big|1-\frac{r^{n+1}}{E(\bar{u}^{n+1})} \big|+\frac{r^{n+1}}{E(\bar{u}^{n+1})}\big|\|h[u(t^{n+1})]\|^2-\|h(\bar{u}^{n+1})\|^2 \big|\\
&:=P^n_1+P^n_2.
\end{split}
\end{equation}
For $P^n_1$, it follows from \eqref{eq: H2high}, $E(v)>\underbar{C}>0\, ,\forall v$ and Theorem \ref{stableThm}  that
\begin{equation}\label{eq: P1}
\begin{split}
P^n_1 & \le C \big|1-\frac{r^{n+1}}{E(\bar{u}^{n+1})} \big|\\
& \le C \big|\frac{r(t^{n+1})}{E[u(t^{n+1})]}-\frac{r^{n+1}}{E[u(t^{n+1})]} \big|+C \big|\frac{r^{n+1}}{E[u(t^{n+1})]}-\frac{r^{n+1}}{E(\bar{u}^{n+1})} \big|\\
& \le C \big(|E[u(t^{n+1})]-E(\bar{u}^{n+1})|+|s^{n+1}|\big).
\end{split}
\end{equation}
For $P^n_2$, it follows from \eqref{eq: H2high}, \eqref{eq: ghigh}, \eqref{eq: Cbarhigh},\,\eqref{eq: Cbar2high},\, $E(v)>\underbar{C}>0$ and Theorem \ref{stableThm}  that
\begin{equation}\label{eq: P2}
\begin{split}
P^n_2 &\le C \big |\|h(\bar{u}^{n+1})\|^2-\|h[u(t^{n+1})]\|^2 \big | \\
& \le C \|h(\bar{u}^{n+1})-h[u(t^{n+1})]\|(\|h(\bar{u}^{n+1})\|+\|h[u(t^{n+1})]\|)\\
& \le C\bar{C} \big(\|\Delta \bar{e}^{n+1}\|+\lambda\|\bar{e}^{n+1}\|+\|g(\bar{u}^{n+1})-g[u(t^{n+1})]\| \big)\\
& \le C\bar{C} \big(\|\Delta \bar{e}^{n+1}\|+ \|\bar{e}^{n+1}\| \big).
\end{split}
\end{equation}
On the other hand,
\begin{equation}\label{eq: Lun+1}
\begin{split}
|E[u(t^{n+1})]-E(\bar{u}^{n+1})| & \le \frac{1}{2}\big(\|\nabla u(t^{n+1})\|+\|\nabla \bar{u}^{n+1}\| \big)\|\nabla u(t^{n+1})-\nabla \bar{u}^{n+1}\| \\
& + \frac{\lambda}{2}\big(\|u(t^{n+1})\|+\| \bar{u}^{n+1}\| \big)\| u(t^{n+1})- \bar{u}^{n+1}\| \\
&+ \int F[u(t^{n+1})] dx - \int F(\bar{u}^{n+1})dx \\
& \le C\bar{C} \big( \|\nabla \bar{e}^{n+1}\|+ \|\bar{e}^{n+1}\| \big).
\end{split}
\end{equation}
Now, combining \eqref{eq: eH2high}, \eqref{eq: errorxifirst2}- \eqref{eq: Lun+1}, we arrive at
\begin{equation*}
\begin{split}
|s^{m+1}| & \le \delta t\sum_{q=0}^{m}\big |\|h[u(t^{q+1})]\|^2
-\frac{r^{q+1}}{E(\bar{u}^{q+1})}\|h(\bar{u}^{q+1})\|^2 \big|  +\sum_{q=0}^{m}|T_1^q|\\
 & \le  C \delta t \sum_{q=0}^{m}|s^{q+1}|+C\bar{C}\delta t \sum_{q=0}^{m}\|\bar{e}^{q+1}\|_{H^2}
+ C\delta t\int_0^{T}(\|u_t(s)\|^2_{H^1}+\|u_{tt}(s)\|_{H^1}) ds\\
\le & C \delta t \sum_{q=0}^{m}|s^{q+1}| +C\bar{C} \sqrt{C_2(1+C_0^{2k+2})}\delta t^k+C \delta t.
\end{split}
\end{equation*}
Applying the discrete Gronwall lemma to the above inequality with $\delta t<\frac{1}{2C}$, we obtain
\begin{equation}\label{eq: sn+1}
\begin{split}
|s^{m+1}| & \le C \exp((1-C\delta t)^{-1}) \delta t(\bar{C} \sqrt{C_2(1+C_0^{2k+2})}\delta t^{k-1}+1)\\
& \le C_3 \delta t(\bar{C} \sqrt{C_2(1+C_0^{2k+2})}\delta t^{k-1}+1),
\end{split}
\end{equation}
where $C_3$  is independent of $C_0$ and $\delta t$, can be defined as
\begin{equation}\label{C3second}
C_3:=C\exp(2),
\end{equation}
then $\delta t<\frac{1}{2C}$ can be guaranteed by
\begin{equation}\label{cond3}
\delta t <\frac{1}{C_3}.
\end{equation}
Hence, noting \eqref{eq: P1}, \eqref{eq: Lun+1},\eqref{eq: sn+1} and \eqref{eq: Cbarhigh}, we have
\begin{equation}\label{eq: boundxin+1}
\begin{split}
|1-\xi^{m+1}| & \le C \big(|E[u(t^{m+1})]-E(\bar{u}^{m+1})|+|s^{m+1}|\big)\\
& \le C(\bar{C}\|\bar{e}^{m+1}\|_{H^1}+|s^{m+1}|)\\
& \le C\delta t \big(\bar{C} \sqrt{C_2(1+C_0^{2k+2})} \delta t^{k-1} +C_3(\bar{C} \sqrt{C_2(1+C_0^{2k+2})} \delta t^{k-1}+1) \big)\\
&\le C_4\delta t(\sqrt{1+C_0^{2k+2}}\delta t^{k-1}+1),
\end{split}
\end{equation}
where the constant $C_4$   is independent of $C_0$ and $\delta t$. Without loss of generality, we  assume $C_4>\max\{C_2, C_3,1\}$ to simplify the proof below.

As a result of \eqref{eq: boundxin+1},  $|1-\xi^{m+1}|\le C_0 \delta t$ if we define $C_0$ such that
\begin{equation}\label{eq: C_0}
C_4(\sqrt{1+C_0^{2k+2}}\delta t^{k-1}+1)\le C_0.
\end{equation}
For the cases $k \ge 2$, the above can be satisfied if we choose $C_0=3C_4$ and $\delta t \le \frac1{1+C_0^{k+1}}$:
\begin{equation}\label{C0def}
C_4(\sqrt{1+C_0^{2k+2}}\delta t^{k-1}+1)\le C_4[(1+C_0^{k+1})\delta t+1] \le 3C_4=C_0.
\end{equation}
For the case $k=1$, we can not define $C_0$ satisfying \eqref{eq: C_0} if  $\eta_1^{n+1}=1-(1-\xi^{n+1})^2$.
However, if we choose $\eta_1^{n+1}=1-(1-\xi^{n+1})^3$, we  can repeat the same process above and arrive at a similar version of \eqref{eq: C_0} for the first order case:
\begin{equation}\label{C_01st}
C_4(\sqrt{1+C_0^{6}}\delta t+1)\le C_0.
\end{equation}
The above  can be satisfied if we choose $C_0=3C_4$ and $\delta t<\frac{1}{C_0^3}$ so that
\begin{equation*}\label{C0def2}
C_4(\sqrt{1+C_0^{6}\delta t^{2}}+1)\le C_4[1+C_0^{3}\delta t+1] \le  3C_4=C_0.
\end{equation*}
To summarize, under the condition
\begin{equation}\label{cond4}
\delta t \le \frac1{1+C_0^{k+2}}, \quad 1\le k \le 5,
\end{equation}
we have $|1-\xi^{m+1}|\le C_0 \delta t$. Note that with  $C_4>\max\{C_2, C_3,1\}$, \eqref{cond4} also implies \eqref{cond2} and \eqref{cond3}. The induction process for \eqref{eq: prestepsk} is complete.
\medskip

Finally, thanks to \eqref{eq: eH2high}, it remains to show $\|e^{m+1}\|_{H^2}\le C \delta t^k$.

We derive from \eqref{eq: Nsav2} and \eqref{eq: Cbarhigh} that
\begin{equation}\label{eq: errorxin+1h}
 \|u^{m+1}-\bar{u}^{m+1}\|_{H^2} \le |\eta_k^{m+1}-1|
\|\bar u^{m+1}\|_{H^2} \le |\eta_k^{m+1}-1| \bar{C}.
\end{equation}
On the other hand, we derive from \eqref{eq: prestepsk} that
\begin{equation}\label{eq: errorxin+2h}
|\eta_k^{m+1}-1| \le C_0^{k+1} \delta t^{k+1}.
\end{equation}
Then it follows from \eqref{eq: eH2high}, \eqref{eq: errorxin+1h} and \eqref{eq: errorxin+2h} and combine the condition \eqref{eq: dtcond1}, \eqref{cond5} and \eqref{cond4} on $\delta t$ that
\begin{equation*}
\begin{split}
\|e^{m+1}\|_{H^2}^2 & \le 2\|\bar{e}^{m+1}\|_{H^2}^2+2\|u^{m+1}-\bar{u}^{m+1}\|_{H^2}^2\\
& \le 2{C_2(1+C_0^{2(k+1)})} \delta t^{2k} + 2 \bar{C}^2C_0^{2(k+1)} \delta t^{2(k+1)}
\end{split}
\end{equation*}
holds under the condition $\delta t< \min\{\frac{1}{1+2C_0^{k+2}},\frac{1-\tau_k}{3k}\}$. The proof is complete.
\end{proof}

\begin{rem}\label{remeta}
Note that   we set $\eta_1^{n+1}=1-(1-\xi^{n+1})^3$ purely  for technical reasons in the proof. It is clear that $\eta_1^{n+1}=1-(1-\xi^{n+1})^2$ leads to first-order accuracy which is confirmed
by our numerical tests.
\end{rem}

\section{Error analysis for Cahn-Hilliard type equations}
In this section,
we consider the Cahn-Hilliard type equation
\begin{equation}\label{eq: CH}
\frac{\partial u}{\partial t}=-\Delta^2 u+\lambda \Delta u-\Delta g(u) \quad (\bm x,t)\in \Omega\times (0,T],
\end{equation}
where $\Omega$ is an open bounded domain in $\mathbb{R}^d\,(d=1,2,3)$, with the initial condition $u(\bm x,0)= u^0(\bm x)$ and  boundary conditions
\begin{equation}\label{CHBC}
\text{periodic}, \,\text{or}, \, \frac{\partial u}{\partial \bm n}|_{\partial\Omega}=\frac{\partial \Delta u}{\partial \bm n}|_{\partial\Omega}=0.
\end{equation}
The above equation is a special case of \eqref{dissE} with  $\mathcal{A}=\Delta^2-\lambda\Delta$ and $g(u)$ replaced by $-\Delta g(u)$. It   satisfies the dissipation law \eqref{dissL} with $E(u)=\frac12(\mathcal{L}u,u)+(G(u),1)$ where $(\mathcal{L}u,u)=(\nabla u,\nabla u)+\lambda (u,u)$,  $G(u)=\int^u g(v)dv$ and $\mathcal{K}(u)=(\nabla \frac{\delta E}{\delta u}, \nabla \frac{\delta E}{\delta u})$.

In particular, with $g(u)=(1-u^2)u$ and $\lambda=0$, the above equation becomes the celebrated Cahn-Hilliard equation \cite{Cah.H58}.

We first recall the following result (cf. for instance \cite{temam2012}).
\begin{thm}\label{prop}
Let $u^0 \in H^2$, and \eqref{eq: 2.6} holds. We assume additionally
\begin{equation}\label{eq: 2.7}
|g''(x)|< C(|x|^{p'}+1),\quad  p'>0 \,\,  arbitrary \quad  \,if \,  \,n=1,2;\quad  \, 0<p'<3 \quad if \, n=3.
\end{equation}
Then for any $T>0$, there exists a unique solution $u$ for  \eqref{eq: CH} such that
\begin{equation*}
u\in C([0,T];H^2) \cap L^2(0,T;H^4).
\end{equation*}
\end{thm}

We also recall the following result (see Lemma 2.3 in \cite{SX18}) which we shall use to deal with the nonlinear term.

\begin{lemma}\label{lemma2}
Assume that $\|u\|_{H^1} \le M$, and that \eqref{eq: 2.6} and \eqref{eq: 2.7} hold.
Then for any $u \in H^4$, there exist $0 \le \sigma <1$ and a constant $C(M)$ such that the following inequality holds:
\begin{equation*}
\|\Delta g(u)\|^2 \le C(M)(1+\|\Delta^2 u\|^{2\sigma}).
\end{equation*}
\end{lemma}

For \eqref{eq: CH}, the $k$th-order version of
 \eqref{eq: Nsav1} and \eqref{Nsav4} read:
 \begin{equation}\label{eq: ubarCH}
\frac{\alpha_k \bar{u}^{n+1}-A_k(u^n)}{ \delta t}=-\Delta\big(\Delta \bar{u}^{n+1}-\lambda \bar{u}^{n+1}+g[B_k(\bar{u}^n)]\big),
\end{equation}
and
\begin{equation}
\frac{\alpha_k u^{n+1}-\eta^{n+1}_{k}A_k(u^n)}{ \delta t}=-\Delta\big(\Delta u^{n+1}-\lambda u^{n+1}+\eta_k^{n+1}g[B_k(\bar{u}^n)]\big),\label{eq: NsavCH1} 
\end{equation}
where $\alpha_k$, $A_k$, $B_k$ defined in \eqref{eq:bdf3} - \eqref{eq:bdf5}.

\begin{theorem}\label{ThmCH} Given initial condition $\bar{u}^0=u^0=u(0)$,  $r^0=E[u^0]$.
Let $\bar u^{n+1}$ and $u^{n+1}$ be computed with the $k$th order scheme  \eqref{eq: Nsav1}- \eqref{eq: Nsav2} $(1 \le k \le5)$ for \eqref{eq: CH} with
\begin{equation*}
 \eta_1^{n+1}=1-(1-\xi^{n+1})^{3}, \quad \eta_k^{n+1}=1-(1-\xi^{n+1})^{k+1}\; (k=2,3,4,5).
\end{equation*}
We assume \eqref{eq: 2.6} and \eqref{eq: 2.7} hold, and
\begin{equation*}
u \in C([0,T];H^3),\; \frac{\partial ^j u}{\partial t^j} \in L^2(0,T;H^2) \,\, 1\le j \le k,\;
\frac{\partial ^{k+1} u}{\partial t^{k+1}} \in L^2(0,T;L^2).
\end{equation*}
Then for $n+1 \le T/\delta t$ and $\delta t \le \text{min}\{\frac{1}{1+4C_0^{k+2}},\frac{1-\tau_k}{3k}\}$, we have
\begin{equation*}
\|\bar e^{n+1}\|_{H^2},\; \|e^{n+1}\|_{H^2}\le C \delta t^{k},
\end{equation*}
where the constants $C_0$, $C$ are dependent on $T,\, \Omega,$ the $k\times k$ matrix $G=(g_{ij})$ in Lemma \ref{lemmaODEH} and the exact solution $u$ but are independent of $\delta t$.
\end{theorem}

Since the proof of this theorem shares some similar procedures with the proof of Theorem \ref{ThmAC}, we  shall defer its proof   to the appendix.

\section{Concluding remarks}
We constructed a class of implicit-explicit BDF$k$ SAV schemes, based on the schemes in \cite{HSY20}, for general linear systems. This class of schemes enjoys the following  advantages: (i)
it only requires solving, in most common situations,  one linear system with constant coefficients at each time step, which is the same as the usual IMEX schemes; (ii) it is not restricted to gradient flows and is applicable to  general dissipative systems; and (iii) it can be high-order with
 unconditional stability and suitable for  adaptive time stepping without restriction on time step size; and most importantly, (iv) it leads to a unconditional uniform bound for the numerical solution, for any order $k$ on the norm based on the principal linear term in the energy functional, which is of critical importance for the convergence and error analysis. We presented numerical results which validated the stability and convergence rates of our schemes, and showed that the SAV scheme is at least as accurate as the usual IMEX scheme, and may lead to more accurate solutions in some critical situations (solutions with large gradients or near singularities). 

Using the uniform bound on the norm based on the principal linear operator that we derived for the  BDF$k$ SAV schemes and to a stability result in \cite{nevanlinna1981} for the BDF$k$ $(k=1,2,3,4,5)$ schemes, we were able to establish, with a delicate inductive argument,  rigorous error estimates for the BDF$k$  $(k=1,2,3,4,5)$ SAV schemes in a unified form for the typical Allen-Cahn and Cahn-Hilliard type equations. 

As mentioned in Remark 2, we can replace the BDK$k$ scheme in \eqref{eq: Nsav1} by  other IMEX multistep schemes, and the stability result in Theorem 1 will still hold. However, error analysis for  other implicit-explicit multistep SAV schemes needs to be investigated separately.
\begin{appendix}
\renewcommand{\theequation}{A.\arabic{equation}}
\section{Proof of Theorem \ref{ThmCH}}

For the sake of brevity, we shall only carry out in detail the  error analysis for the first-order case. The analysis for the higher-order cases can be carried out by combining the procedures for the first-order case below and  for the high-order cases in the proof Theorem \ref{ThmAC}. The detail will be left for the interested readers.

As in the proof of Theorem \ref{ThmAC}, we will first prove the following by induction:
\begin{equation}\label{eq: CHpresteps}
|1-\xi^q| \le C_0\,\delta t,\,\, \forall  q\le T/{\delta t},
\end{equation}
where  the constant $C_0$ is dependent on $T,\, \Omega$ and the exact solution $u$ but is independent of $\delta t$, and  will be defined in the proof process.

Under the assumptions, \eqref{eq: CHpresteps}  certainly holds for $q=0$. Now
suppose we have
\begin{equation}\label{eq: CHpresteps2}
|1-\xi^q| \le C_0\,\delta t,\,\, \forall q \le m,
\end{equation}
we shall prove below that \eqref{eq: CHpresteps} holds for $q=m+1$, namely,
\begin{equation}\label{eq: CHxi}
|1-\xi^{m+1}| \le C_0 \delta t.
\end{equation}
We will carry out this proof in three steps.

\textbf{Step 1:  $H^2$ bound for ${u}^{n}$ and $\bar{u}^{n}$ for all $n\le m$.} It follows from Theorem \ref{stableThm} and under condition
\begin{equation}\label{eq: CHdtcond1}
 \delta t \le \min\{\frac{1}{4C_0^3},1\},
\end{equation}
we have
\begin{equation}\label{eq: erroreta}
\frac{3}{4} \le |\eta_1^q| \le 2,\,\,|1-\eta_1^q| \le \frac{\delta t^2}{4}, \forall q\le m,
\end{equation}
and
\begin{equation}\label{eq: CHH1bound}
\|{u}^{q}\|_{H^1} \le M_2,\,\,\forall q\le T/\delta t,\,\,\|\bar{u}^{q}\|_{H^1} \le \frac{4}{3}M_2,\,\forall q\le m.
\end{equation}

Now, consider \eqref{eq: NsavCH1} at step $q$:
\begin{equation}\label{eq: NsavCH1n}
\frac{ u^{q}-\eta_1^{q} u^{n-1}}{ \delta t}=-\Delta^2 u^{q}+ \lambda \Delta u^{q}-\eta_1^{q}\Delta g[\bar{u}^{q-1}]
\end{equation}
Multiply \eqref{eq: NsavCH1n} with $\Delta^2 {u}^{q}$, and by the similar process as step 1 in Theorem \ref{ThmAC}, we can obtain
\begin{equation}\label{eq: CH3.9}
\|\Delta {u}^{q}\|^2-\|\Delta u^{q-1}\|^2+\delta t\| \Delta^2 \bar{u}^{q}\|^2-\frac{\delta t}{2}\| \Delta^2 \bar{u}^{q-1}\|^2 \le C(M_2) \delta t + |1-\eta_1^q|\|u^{q-1}\|^2
\end{equation}

Taking the sum from $0$ to $n\,(\le m)$ of \eqref{eq: CH3.9}, we obtain
\begin{equation*}
\begin{split}
\|\Delta {u}^{n}\|^2+\frac{\delta t}{2} \sum_{q=0}^{n}\| \Delta^2 \bar{u}^{q}\|^2 &\le C(M_2)T+C(u^0)+\delta t^2\sum_{q=1}^{n-1}\|{u}^q\|^2\\
& \le C(M_2)T+C(u^0)+\delta t T M_2^2.
\end{split}
\end{equation*}
with  $C(M_2)$ is a constant only depends on $M_2$ and $C(u^0)$ only depends on $u^0$. Then together with \eqref{eq: CHH1bound} implies
\begin{equation}\label{eq: CHuH2bound}
\|u^n\|_{H^2} \le \sqrt{C(M)T+C(u^0)+ T M_2^2}+M_2:=C_1,\quad\forall n\le m.
\end{equation}
As $ \|u^n\|_{H^2}=\eta_1^n\|\bar{u}^n\|_{H^2}$, \eqref{eq: erroreta} implies
\begin{equation}\label{eq: CHbaruH2bound}
\|\bar{u}^n\|_{H^2} \le \frac{4}{3} C_1,\quad\forall n\le m.
\end{equation}

\textbf{Step 2: estimates for  $\|\bar{e}^{n+1}\|_{H^2}$ and  $\|\bar{e}^{n+1}\|_{H^3}$ for all $0\le n\le m$.} 
By given assumption on the exact solution $u$ and  \eqref{eq: CHbaruH2bound},
 we can choose  $C$ large enough such that
\begin{equation}\label{eq: CHH2first}
\|u(t)\|_{H^3}\le C,\,\, \forall t \le T, \|\bar{u}^q\|_{H^2} \le C,\,\,\forall q\le m,
\end{equation}
and since $H^2 \subset L^{\infty}$, without loss of generality, we can adjust $C$ such that
\begin{equation}\label{eq: CHgfirst}
|g^{(i)}[u(t)]|_{L^\infty}\le C, \,\, \forall t \le T;\,|g^{(i)}(\bar{u}^q)|_{L^\infty} \le C, \,\, \forall q \le m;\, i=0,1,2,3.
\end{equation}

From \eqref{eq: ubarCH}, we can write down the equation for error as
\begin{equation}\label{eq: CHerrorfirst}
\bar{e}^{n+1}-\bar{e}^n=(\eta_1^n-1)\bar{u}^n-\delta t \Delta^2 \bar{e}^{n+1}+\lambda \delta t \Delta \bar{e}^{n+1}+ R^n_1 +\delta t\Delta R^n_2,
\end{equation}
where  $R^n_1$,\, $R^n_2$ are given by
\begin{equation}\label{eq: CHR1first}
R^n_1=u(t^n)-u(t^{n+1})+\delta t u_t(t^{n+1})=\int_{t^n}^{t^{n+1}}(s-t^n)u_{tt}ds,
\end{equation}
and
\begin{equation}\label{eq: CHR2first}
R^n_2=- g(\bar{u}^n)+ g[u(t^{n+1})].
\end{equation}
Taking inner product with $\bar{e}^{n+1}-\Delta \bar{e}^{n+1}+\Delta ^2 \bar{e}^{n+1}$ on both sides of  \eqref{eq: CHerrorfirst}, we obtain
\begin{equation}\label{eq: CHerrorfirst2}
\begin{split}
\frac{1}{2}& \big( \|\bar{e}^{n+1}\|^2-\|\bar{e}^{n}\|^2 \big) +\frac{1}{2}\|\bar{e}^{n+1}-\bar{e}^n\|^2+\delta t \|\Delta \bar{e}^{n+1}\|^2+\lambda \delta t \|\nabla \bar{e}^{n+1}\|^2\\
& +\frac{1}{2} \big( \|\nabla \bar{e}^{n+1}\|^2-\|\nabla \bar{e}^{n}\|^2 \big) +\frac{1}{2}\|\nabla (\bar{e}^{n+1}-\bar{e}^n)\|^2+\delta t \|\nabla \Delta \bar{e}^{n+1}\|^2+\lambda \delta t \|\Delta \bar{e}^{n+1}\|^2\\
& +\frac{1}{2} \big( \|\Delta \bar{e}^{n+1}\|^2-\|\Delta \bar{e}^{n}\|^2 \big) +\frac{1}{2}\|\Delta (\bar{e}^{n+1}-\bar{e}^n)\|^2+\delta t \|\Delta^2 \bar{e}^{n+1}\|^2+\lambda \delta t \|\nabla \Delta \bar{e}^{n+1}\|^2\\
&=(\eta_1^n-1)\big(\bar{u}^n, \bar{e}^{n+1}\big )+\big(R^n_1, \bar{e}^{n+1}\big)-\delta t\big (\nabla R^n_2, \nabla \bar{e}^{n+1} \big)\\
& +(\eta_1^n-1)\big(\nabla \bar{u}^n, \nabla \bar{e}^{n+1}\big )+\big(R^n_1, -\Delta \bar{e}^{n+1}\big)+ \delta t\big (\nabla R^n_2, \nabla \Delta \bar{e}^{n+1} \big)\\
&+(\eta_1^n-1)\big(\Delta \bar{u}^n, \Delta \bar{e}^{n+1}\big )+\big(R^n_1, \Delta^2 \bar{e}^{n+1}\big)+ \delta t\big (\Delta R^n_2, \Delta^2 \bar{e}^{n+1} \big).
\end{split}
\end{equation}
In the following, we bound the right hand side of  \eqref{eq: CHerrorfirst2}. Noting that $|\eta_1^n-1|\le C_0^3\, \delta t^3$, hence
\begin{equation}\label{eq: CHetafirst}
|(\eta_1^n-1)\big(\bar{u}^n, \bar{e}^{n+1}\big )| \le \frac{\|(\eta_1^n-1)\bar{u}^n\|^2}{\delta t} +\frac{\delta t}{4}\|\bar{e}^{n+1}\|^2 \le C C_0^6 \delta t^5 +\frac{\delta t}{4}\|\bar{e}^{n+1}\|^2,
\end{equation}
\begin{equation}\label{eq: CHetafirst2}
|(\eta_1^n-1)\big(\nabla \bar{u}^n, \nabla \bar{e}^{n+1}\big )| \le C C_0^6 \delta t^5 +\frac{\delta t}{4}\|\nabla \bar{e}^{n+1}\|^2,
\end{equation}
and
\begin{equation}\label{eq: CHetafirst3}
|(\eta_1^n-1)\big(\Delta \bar{u}^n, \Delta \bar{e}^{n+1}\big )| \le C C_0^6 \delta t^5 +\frac{\delta t}{4}\|\Delta \bar{e}^{n+1}\|^2.
\end{equation}
It follows from \eqref{eq: CHR1first} that
\begin{equation}\label{eq: CHR1error}
\|R^n_1\|^2 \le C \delta t^3 \int_{t^n}^{t^{n+1}} \|u_{tt}(s)\|^2 ds.
\end{equation}
Therefore,
\begin{equation}\label{eq: CHR1errorfirst}
|\big(R^n_1, \bar{e}^{n+1}\big)| \le \frac{1}{2\delta t}\|R^n_1\|^2+ \frac{\delta t}{2} \|\bar{e}^{n+1}\|^2 \le \frac{\delta t}{2} \|\bar{e}^{n+1}\|^2+C \delta t^2 \int_{t^n}^{t^{n+1}} \|u_{tt}(s)\|^2 ds,
\end{equation}
\begin{equation}\label{eq: CHR1errorfirst2}
|\big(R^n_1, -\Delta \bar{e}^{n+1}\big)| \le \frac{\delta t}{2} \|\Delta \bar{e}^{n+1}\|^2+C \delta t^2 \int_{t^n}^{t^{n+1}} \|u_{tt}(s)\|^2 ds,
\end{equation}
and
\begin{equation}\label{eq: CHR1errorfirst3}
|\big(R^n_1, \Delta^2 \bar{e}^{n+1}\big)| \le \frac{\delta t}{2} \|\Delta^2 \bar{e}^{n+1}\|^2+C \delta t^2 \int_{t^n}^{t^{n+1}} \|u_{tt}(s)\|^2 ds.
\end{equation}
Noting that
\begin{equation}\label{eq: CHnablaR2}
\begin{split}
|\nabla R^n_2|&=|\nabla g(\bar{u}^n)-\nabla g[u(t^n)]+\nabla g[u(t^n)]-\nabla g[u(t^{n+1})]|\\
& \le |g'(\bar{u}^n)\nabla \bar{u}^n-g'[u(t^n)]\nabla u(t^n)|+|g'[u(t^n)]\nabla u(t^n)-g'[u(t^{n+1})]\nabla u(t^{n+1})|\\
& \le |g'(\bar{u}^n)|\big|\nabla \bar{u}^n-\nabla u(t^n)\big|+\big|g'(\bar{u}^n)-g'[u(t^n)]\big|\big|\nabla u(t^n)\big|\\
& +\big |g'[u(t^n)]-g'[u(t^{n+1})]\big|\big|\nabla u(t^n)\big|+\big|g'[u(t^{n+1})]\big|\big|\nabla u(t^n)-\nabla u(t^{n+1})\big|\\
& \le C\big (|\nabla \bar{e}^{n}|+|\bar{e}^n|+\int_{t^n}^{t^{n+1}}\big(|u_t(s)|+|\nabla u_t(s)| \big) ds \big),
\end{split}
\end{equation}
then for the terms with $\nabla R^n_2$, it follows from \eqref{eq: CHnablaR2} that
\begin{equation}\label{eq: CHR2error1}
\begin{split}
\delta t|(\nabla R^n_2, \nabla \bar{e}^{n+1})| &\le \frac{\delta t}{2}\|\nabla R^n_2\|^2+\frac{\delta t}{2}\|\nabla \bar{e}^{n+1}\|^2\\
& \le C \delta t(\|\nabla \bar{e}^{n+1}\|^2+\| \bar{e}^{n}\|_{H^1}^2)+C\delta t^2 \int_{t^n}^{t^{n+1}}\|u_t(s)\|^2_{H^1}ds,
\end{split}
\end{equation}
and
\begin{equation}\label{eq: CHR2error2}
\begin{split}
\delta t|(\nabla R^n_2, \nabla\Delta \bar{e}^{n+1})| &\le \frac{\delta t}{2}\|\nabla R^n_2\|^2+\frac{\delta t}{2}\|\nabla \Delta \bar{e}^{n+1}\|^2\\
& \le C \delta t\| \bar{e}^{n}\|_{H^1}^2+C\delta t^2 \int_{t^n}^{t^{n+1}}\|u_t(s)\|^2_{H^1}ds+\frac{\delta t}{2}\|\nabla \Delta \bar{e}^{n+1}\|^2.
\end{split}
\end{equation}
For the term with $\Delta R^n_2$, since
\begin{equation*}
|\Delta R^n_2|\le |-\Delta g(\bar{u}^n)+\Delta g[u(t^n)]|+|-\Delta g[u(t^n)]+\Delta g[u(t^{n+1})]|:=Q^n_1+Q^n_2,
\end{equation*}
and note that
\begin{equation*}
\Delta g(u)=g''(u)|\nabla u|^2+g'(u)\Delta u,
\end{equation*}
by using \eqref{eq: CHH2first} and \eqref{eq: CHgfirst}, we have
\begin{equation*}
\begin{split}
Q^n_1& \le \big |g''(\bar{u}^n)(|\nabla \bar{u}^n|^2-|\nabla u(t^n)|^2)\big |+\big| |\nabla u(t^n)|^2(g''(\bar{u}^n)-g''[u(t^n)]) \big |\\
& +|g'(\bar{u}^n)(\Delta \bar{u}^n-\Delta u(t^n))|+|\Delta u(t^n)(g'(\bar{u}^n)-g'[u(t^n)])|\\
& \le C \big( |\nabla \bar{e}^n|+|\bar{e}^n|+|\Delta \bar{e}^{n}|),
\end{split}
\end{equation*}
and
\begin{equation*}
Q^n_2 \le C\big(\int_{t^n}^{t^{n+1}}|\nabla u_t(s)|^2 ds+\int_{t^n}^{t^{n+1}}|\Delta u_t(s)|ds \big ).
\end{equation*}
Therefore,
\begin{equation}\label{eq: CHR2error3}
\begin{split}
\delta t |\big (\Delta R^n_2, \Delta ^2 \bar{e}^{n+1} \big)| & \le \delta t |\big (Q^n_1,\Delta^2 \bar{e}^{n+1} \big)|+\delta t |\big (Q^n_2,\Delta^2 \bar{e}^{n+1} \big)|\\
& \le \delta t \|Q^n_1\|^2+\frac{\delta t}{4} \|\Delta^2 \bar{e}^{n+1}\|^2+\delta t \|Q^n_2\|^2+\frac{\delta t}{4} \|\Delta^2 \bar{e}^{n+1}\|^2\\
& \le C \delta t (\|\bar{e}^n\|^2+\|\nabla \bar{e}^n\|^2+\|\Delta \bar{e}^n\|^2)+\frac{\delta t}{2} \|\Delta^2 \bar{e}^{n+1}\|^2 \\
& + C \delta t^2 \int_{t^n}^{t^{n+1}}\| u_t(s)\|_{H^2}^2ds,
\end{split}
\end{equation}
where we used the following inequality
\begin{equation*}
\begin{split}
\int_{\Omega}\big(\int_{t^n}^{t^{n+1}}(|\nabla u_t(s)|+|\Delta u_t(s)|)ds\big)^2 d \bm x & \le \int_{\Omega}\big(\int_{t^n}^{t^{n+1}}(|\nabla u_t(s)|+|\Delta u_t(s)|)^2 ds \int_{t^n}^{t^{n+1}}1ds\big) d\bm x \\
& \le C \delta t \int_{t^n}^{t^{n+1}} \|u_t(s)\|_{H^2}^2 ds.
\end{split}
\end{equation*}
Now, combining \eqref{eq: CHerrorfirst2}-\eqref{eq: CHR2error2} and \eqref{eq: CHR2error3} and dropping some unnecessary terms, we arrive at
\begin{equation}\label{eq: CHbaruH2}
\begin{split}
 \|\bar{e}^{n+1}\|^2-\|\bar{e}^{n}\|^2 &+\|\nabla \bar{e}^{n+1}\|^2-\|\nabla \bar{e}^{n}\|^2 + \|\Delta \bar{e}^{n+1}\|^2-\|\Delta \bar{e}^{n}\|^2+\delta t \|\nabla \Delta \bar{e}^{n+1}\|^2\\
& \le CC_0^6\delta t^5+C \delta t (\|\nabla \bar{e}^{n+1}\|^2+\|\bar{e}^{n+1}\|^2+\|\Delta \bar{e}^{n}\|^2+\|\nabla \bar{e}^{n}\|^2+\| \bar{e}^{n}\|^2)\\
& + C\delta t^2 \int_{t^n}^{t^{n+1}}(\|u_t(s)\|_{H^2}^2+\|u_{tt}(s)\|^2)ds.
\end{split}
\end{equation}
Taking the sum of the above for $n$ from $0$ to $m$, we obtain
\begin{equation}\label{eq: CHbaruH3}
\|\bar{e}^{m+1}\|_{H^2}^2+\delta t \sum_{q=0}^{m}\|\nabla \Delta \bar{e}^{q+1}\|^2 \le C \delta t \sum_{q=0}^{m+1}\|\bar{e}^q\|_{H^2}^2+ C\delta t^2 \int_{0}^{T}(\|u_t(s)\|_{H^2}^2+\|u_{tt}(s)\|^2+C_0^6 \delta t^2 )ds.
\end{equation}
Finally, we can obtain the following estimate for $\bar{e}^{m+1}$ by applying the discrete Gronwall's inequality to \eqref{eq: CHbaruH3} with $\delta t< \frac{1}{2C}$:
\begin{equation}\label{eq: CHerrorbare}
\begin{split}
\|\bar{e}^{n+1}\|_{H^2}^2+\delta t \sum_{q=0}^{n}\|\nabla \Delta \bar{e}^{q+1}\|^2 &\le C\exp((1-\delta t C)^{-1}) \delta t^2 \int_0^{T}(\|u_t(s)\|_{H^2}^2+\|u_{tt}(s)\|^2+C_0^6 \delta t^2 )ds\\
& \le C_2(1+C_0^6 \delta t^2)\delta t^2,\quad \forall\, 0\le n\le m.
\end{split}
\end{equation}
where $C_2$ is independent of $\delta t$ and $C_0$, can be defined as
\begin{equation}\label{C2H1}
C_2:=C\exp(2)\max\big(\int_0^{T}(\|u_t(s)\|_{H^2}^2+\|u_{tt}(s)\|^2)ds, 1 \big),
\end{equation}
and hence $\delta t< \frac{1}{2C}$ can be guaranteed by
$\delta t< \frac{1}{C_2}$.
 In particular, \eqref{eq: CHerrorbare} implies
\begin{equation}\label{eq: CHeH2}
\|\bar{e}^{n+1}\|_{H^2},\,\,\big(\delta t \sum_{q=0}^{n}\|\nabla \Delta \bar{e}^{q+1}\|^2 \big)^{1/2}\le \sqrt{C_2(1+C_0^6\delta t^2)} \delta t,\quad \forall\, 0\le n\le m.
\end{equation}
Combining \eqref{eq: CHH2first} and \eqref{eq: CHeH2}, we obtain that for all $ \forall\, 0\le n\le m$ and under the condition on $
\delta t$ in \eqref{eq: CHdtcond1}, we have
\begin{equation}\label{eq: CHCbar}
\|\bar{u}^{n+1}\|_{H^2},\,\,\big(\delta t \sum_{q=0}^{n}\|\nabla \Delta \bar{u}^{q+1}\|^2 \big)^{1/2} \le \sqrt{C_2(1+C_0^6\delta t^2)}\delta t+ C \le \sqrt{C_2(1+1)}+C:=\bar{C}.
\end{equation}
Note that $H^2 \subset L^{\infty}$, without loss of generality, we can adjust
$\bar C$ so that  we have
\begin{equation}\label{eq: CHCbar2}
\|g(\bar{u}^{n+1})\|,\,\|g'(\bar{u}^{n+1})\| \le \bar C,\quad \forall 0\le n\le m.
\end{equation}

\textbf{Step 3: estimate for $|1-\xi^{n+1}|$.}
It follows from \eqref{eq: Nsav3} that the equation for the error $\{s^j\}$ can be written as
\begin{equation}\label{eq: CHerrorxifirst}
s^{n+1}-s^{n}=
\delta t\big(\|\nabla h[u(t^{n+1})]\|^2-\frac{r^{n+1}}{E(\bar{u}^{n+1})}\|\nabla h(\bar{u}^{n+1})\|^2 \big) +T_1^n,
\end{equation}
where $h(u)=\frac{\delta E}{\delta u}=-\Delta u +\lambda u- g(u)$  and truncation errors $T_1^n$ is given in \eqref{eq: T1} with a bound given in \eqref{eq: boundT1}.

Taking the sum of  \eqref{eq: CHerrorxifirst} for $n$ from $0$ to $m$, since $s^0=0$, we have
\begin{equation*}\label{eq: CHerrorxifirst2}
s^{m+1}=
\delta t\sum_{q=0}^{m}\big(\|\nabla h[u(t^{q+1})]\|^2-\frac{r^{q+1}}{E(\bar{u}^{q+1})}\|\nabla h(\bar{u}^{q+1})\|^2 \big) +\sum_{q=0}^{m}T_1^q.
\end{equation*}
For $\|\nabla h[u(t^{n+1})]\|^2-\frac{r^{n+1}}{E(\bar{u}^{n+1})}\|\nabla h(\bar{u}^{n+1})\|^2 $, we have
\begin{equation*}\label{eq: K1K2}
\begin{split}
\big |&\|\nabla h[u(t^{n+1})]\|^2-\frac{r^{n+1}}{E(\bar{u}^{n+1})}\|\nabla h(\bar{u}^{n+1})\|^2 \big |\\
& \le \|\nabla h[u(t^{n+1})]\|^2\big|1-\frac{r^{n+1}}{E(\bar{u}^{n+1})} \big|+\frac{r^{n+1}}{E(\bar{u}^{n+1})}\big|\|\nabla h[u(t^{n+1})]\|^2-\|\nabla h(\bar{u}^{n+1})\|^2 \big|\\
&:=K^n_1+K^n_2.
\end{split}
\end{equation*}
For $K^n_1$, it follows from \eqref{eq: CHH2first},  $E(\bar{u}^{n+1})>\underbar{C}>0$ and Theorem \ref{stableThm}  that
\begin{equation*}\label{eq: K1}
\begin{split}
K^n_1 & \le C \big|1-\frac{r^{n+1}}{E(\bar{u}^{n+1})} \big|\\
& =C \big|\frac{r(t^{n+1})}{E[u(t^{n+1})]}-\frac{r^{n+1}}{E[u(t^{n+1})]} \big|+C \big|\frac{r^{n+1}}{E[u(t^{n+1})]}-\frac{r^{n+1}}{E(\bar{u}^{n+1})} \big|\\
& \le C \big(|E[u(t^{n+1})]-E(\bar{u}^{n+1})|+|s^{n+1}|\big).
\end{split}
\end{equation*}
For $K^n_2$, it follows from \eqref{eq: CHH2first}, \eqref{eq: CHgfirst}, \eqref{eq: CHCbar}, \eqref{eq: CHCbar2}, $E(\bar{u}^{n+1})>\underbar{C}>0$ and Theorem \ref{stableThm} that
\begin{equation*}\label{eq: K2}
\begin{split}
K^n_2 & \le C \big |\|\nabla h(\bar{u}^{n+1})\|^2-\|\nabla h[u(t^{n+1})]\|^2 \big | \\
& \le C \|\nabla h(\bar{u}^{n+1})-\nabla h[u(t^{n+1})]\|(\|\nabla h(\bar{u}^{n+1})\|+\|\nabla h[u(t^{n+1})]\|)\\
& \le C\bar{C}(1+\|\nabla \Delta \bar{u}^{n+1}\|) \big(\|\nabla \Delta \bar{e}^{n+1}\|+\lambda\|\nabla \bar{e}^{n+1}\|+\|\nabla(g(\bar{u}^{n+1})-g[u(t^{n+1})])\| \big)\\
& \le C\bar{C}\big(\|\nabla \Delta \bar{e}^{n+1}\|+ \|\nabla \bar{e}^{n+1}\| \big)+C\bar{C}\|\nabla \Delta \bar{u}^{n+1}\|\|\nabla \Delta \bar{e}^{n+1}\|+C\bar{C}\|\nabla \Delta \bar{u}^{n+1}\|\|\nabla \bar{e}^{n+1}\|.
\end{split}
\end{equation*}
It then follows from \eqref{eq: CHeH2}, \eqref{eq: CHCbar}  and the  Cauchy-Schwarz inequality that
\begin{equation*}
\delta t \sum_{q=1}^{n+1}\|\nabla \Delta \bar{u}^q\|\|\nabla \bar{e}^{q}\| \le \big(\delta t\sum_{q=1}^{n+1}\|\nabla \Delta \bar{u}^q\|^2 \delta t\sum_{q=1}^{n+1} \|\nabla \bar{e}^{q}\|^2 \big)^{1/2} \le C\bar{C}\sqrt{C_2(1+C_0^6\delta t^2)} \delta t,
\end{equation*}
and
\begin{equation*}
\delta t \sum_{q=1}^{n+1}\|\nabla \Delta \bar{u}^q\|\|\nabla \Delta \bar{e}^{q}\| \le \big(\delta t\sum_{q=1}^{n+1}\|\nabla \Delta \bar{u}^q\|^2 \delta t\sum_{q=1}^{n+1} \|\nabla \Delta \bar{e}^{q}\|^2 \big)^{1/2} \le C\bar{C}\sqrt{C_2(1+C_0^6\delta t^2)} \delta t.
\end{equation*}
For $E[u(t^{n+1})]-E (\bar{u}^{n+1})$, we have estimate \eqref{eq: Lun+1}.

Now, we are ready to estimate $s^{m+1}$. Combine the estimate obtained above, \eqref{eq: CHerrorxifirst2} leads to
\begin{equation}\label{eq: CHsn+1}
\begin{split}
|s^{m+1}| & \le
\delta t\sum_{q=0}^{m}\big |\|\nabla h[u(t^{q+1})]\|^2-\frac{E_0(\bar{u}^{q+1})+r^{q+1}}{E(\bar{u}^{q+1})}\|\nabla h(\bar{u}^{q+1})\|^2 \big|
+\sum_{q=0}^{m}|T_1^q|\\
& \le C\delta t \sum_{q=0}^{m}|s^{q+1}|+C\bar{C}\delta t \sum_{q=0}^{m}\|\bar{e}^{q+1}\|_{H^1}+C\bar{C}\delta t \sum_{q=0}^{m}\|\nabla \Delta {e}^{q+1}\|\\
& +C\bar{C}\delta t \sum_{q=1}^{m+1}\|\nabla \Delta \bar{u}^q\|\|\nabla \bar{e}^{q}\|+C\bar{C}\delta t \sum_{q=1}^{m+1}\|\nabla \Delta \bar{u}^q\|\|\nabla \Delta \bar{e}^{q}\|\\
& + C\delta t\int_0^{t^{m+1}}(\|u_t(s)\|^2_{H^1}+\|u_{tt}(s)\|_{H^1}) ds\\
& \le C \delta t \sum_{q=0}^{m}|s^{q+1}| +C \bar{C}^2 \delta t \big(\sqrt{C_2(1+C_0^6\delta t^2)}+1\big)
\end{split}
\end{equation}
Finally, applying the discrete Gronwall's inequality on \eqref{eq: CHsn+1} with $\delta t< \frac{1}{2C}$ , we obtain the following estimate for $s^{n+1}$:
\begin{equation}\label{eq: CHsn+11}
\begin{split}
|s^{n+1}| & \le C\exp((1-\delta t C)^{-1}) \bar{C}^2 \delta t \big(\sqrt{C_2(1+C_0^6\delta t^2)}+1\big)\\
& \le C_3 \delta t  \big(\sqrt{C_2(1+C_0^6\delta t^2)}+1\big),\forall\, 0\le n\le m,
\end{split}
\end{equation}
where $C_3$ is independent of $\delta t$ and $C_0$, can be defined as
\begin{equation}\label{C3H1}
C_3:=C\bar{C}^2\exp(2).
\end{equation}
Thanks to \eqref{eq: CHsn+11}, we can define $C_0$ and then prove \eqref{eq: CHxi}  by following exactly the same procedure  as {\bf  Step 3 in Theorem \ref{ThmAC}} with the condition
\begin{equation}\label{cond4CH}
\delta t \le \frac1{1+C_0^{3}}
\end{equation}
The induction process for \eqref{eq: CHpresteps} is completed.

Finally, thanks to \eqref{eq: CHeH2}, it remains to show $\|e^{m+1}\|_{H^2}\le C \delta t^k$.

 We derive from  \eqref{eq: CHCbar} that
\begin{equation}\label{eq: CHerrorbar}
 \|u^{m+1}-\bar{u}^{m+1}\|_{H^2} \le |\eta_1^{m+1}-1|
\|\bar u^{m+1}\|_{H^2} \le |\eta_1^{m+1}-1| \bar{C}.
\end{equation}
On the other hand,
\eqref{eq: CHpresteps} implies
\begin{equation}\label{eq: CHerrorxin+1}
|\eta_1^{m+1}-1| \le C_0^3 \delta t^3.
\end{equation}
Then it follows from \eqref{eq: CHeH2}, \eqref{eq: CHerrorbar} and \eqref{eq: CHerrorxin+1} that \begin{equation*}
\begin{split}
\|e^{m+1}\|_{H^2}^2 & \le 2\|\bar{e}^{m+1}\|_{H^2}^2+2\|u^{m+1}-\bar{u}^{m+1}\|_{H^2}^2\\
& \le 2{C_2(1+C_0^6\delta t^2)} \delta t^2 + 2 \bar{C}^2C_0^6 \delta t^6.
\end{split}
\end{equation*}
To summarize, combine the condition \eqref{eq: CHdtcond1} and \eqref{cond4CH} on $\delta t$, we obtain $\|e^{m+1}\|_{H^2}\le C\delta t$ with $\delta t< \frac{1}{1+4C_0^{3}}$. The proof for the case $k=1$ is complete.
\end{appendix}

\bibliographystyle{plain}
\bibliography{bib_sav_error}

\end{document}